\documentclass[reqno,centertags]{amsart}

\usepackage{amssymb,amsmath,amsthm} \usepackage{graphicx}

\def\Xint#1{\mathchoice {\XXint\displaystyle\textstyle{#1}}%
  {\XXint\textstyle\scriptstyle{#1}}%
  {\XXint\scriptstyle\scriptscriptstyle{#1}}%
  {\XXint\scriptscriptstyle\scriptscriptstyle{#1}}%
  \!\int} \def\XXint#1#2#3{{\setbox0=\hbox{$#1{#2#3}{\int}$}
    \vcenter{\hbox{$#2#3$}}\kern-.5\wd0}} \def\dashint{\Xint-}

\newcommand{\R}{\mathbb{R}} \newcommand{\Z}{\mathbb{Z}}
\newcommand{\T}{\mathbb{T}} \newcommand{\C}{\mathbb{C}}
\newcommand{\N}{\mathbb{N}}

\theoremstyle{plain} \newtheorem{theorem}{Theorem}[section]
\newtheorem{lemma}[theorem]{Lemma}
\newtheorem{corollary}[theorem]{Corollary}
\newtheorem{prop}[theorem]{Proposition}

\theoremstyle{definition} \newtheorem{definition}[theorem]{Definition}
\theoremstyle{remark} \newtheorem{remark}{Remark}

\DeclareMathOperator{\supp}{supp} \DeclareMathOperator{\sign}{sign}
 \DeclareMathOperator{\Imag}{Im}

\newcommand{\test}{\mathcal{S}(\R \times \T)}
\newcommand{\eps}{\varepsilon} \newcommand{\lb}{\langle}
\newcommand{\rb}{\rangle}
\newcommand{\ls}{\lesssim}\newcommand{\gs}{\gtrsim}

\begin{document}

\author[A.~Gr{\"u}nrock]{Axel~Gr{\"u}nrock}
\author[S.~Herr]{Sebastian~Herr} \title[Low regularity LWP of the DNLS
with periodic initial data] {Low regularity local well-posedness of
  the Derivative Nonlinear Schr\"odinger Equation with periodic
  initial data}


\address{Axel~Gr{\"u}nrock: Bergische Universit\"at Wuppertal,
  Fachbereich C: Mathematik / Naturwissenschaften, Gau{\ss}stra{\ss}e
  20, 42097 Wuppertal, Germany.}
\email{axel.gruenrock@math.uni-wuppertal.de} \address{Sebastian~Herr:
  Rheinische Friedrich-Wilhelms-Universit\"at Bonn, Mathematisches
  Institut, Beringstra{\ss}e 1, 53115 Bonn, Germany.}

\email{herr@math.uni-bonn.de}

\begin{abstract}
  The Cauchy problem for the derivative nonlinear Schr\"odinger
  equation with periodic boundary condition is considered. Local
  well-posedness for data $u_0$ in the space
  $\widehat{H}^{s}_{r}(\T)$, defined by the norms
  $$
  \|u_0\|_{\widehat{H}^{s}_{r}(\T)} =
  \|\lb\xi\rb^s\widehat{u}_0\|_{\ell^{r'}_{\xi}} \, ,
  $$
  is shown in the parameter range $s\ge \frac{1}{2}$, $2>r>\frac{4}{3}$.
  The proof is based on an adaptation of the gauge
  transform to the periodic setting and an appropriate variant of the
  Fourier restriction norm method.
\end{abstract}
\maketitle
\section{Introduction and main result}\label{sect:intro_main}
\noindent
The Cauchy problem for the derivative nonlinear Schr\"odinger equation
\begin{equation}\label{eq:dnls}
  \tag{DNLS}
  \begin{split}
    i\partial_t u + \partial_x^2 u&=i\partial_x(|u|^2u) \\
    u(0,x)&=u_0 (x)
  \end{split}
\end{equation}
with data $u_0$ in the classical Sobolev spaces $H^s(\R)$ of functions
defined on the real line is known to be locally well-posed for $s \geq
\frac{1}{2}$. This was shown by Takaoka in \cite{Tak99}, where he
improved the earlier $H^1(\R)$-result of Hayashi and Ozawa
\cite{HO92,Hay93,HO94}. His method of proof combines the gauge
transform already used by Hayashi and Ozawa with Bourgain's Fourier
restriction norm method. A counterexample of Biagioni and Linares
\cite{BL01} shows the optimality of Takaoka's result on the
$H^s(\R)$-scale of data spaces: For $s<\frac{1}{2}$ the Cauchy problem
\eqref{eq:dnls} is ill-posed in the $C^0$-uniform sense, although the
standard scaling argument suggests local well-posedness for $s>0$.
This gap of $1/2$ derivative between the scaling prediction and
Takaoka's result can be closed by leaving the $H^s(\R)$-scale and
considering data in the spaces $\widehat{H}^{s}_{r}(\R)$ defined by
the norms
\begin{equation*}
  \|u_0\|_{\widehat{H}^{s}_{r}(\R)} = \|\lb\xi\rb^s\widehat{u}_0\|_{L^{r'}_{\xi}},
  \qquad
  \lb\xi\rb=(1+\xi^2)^{\frac{1}{2}},
  \qquad \frac{1}{r}+\frac{1}{r'}=1.
\end{equation*}
We remark that these spaces coincide with $B_{r',k}$ (with weight
$k(\xi)=\lb\xi\rb^s$) introduced by H\"ormander, cf. \cite{Hoe83},
Section 10.1. The idea to consider them as data spaces for nonlinear
Schr\"odinger equations goes back to the work of Cazenave, Vega, and
Vilela \cite{CVV01}, where corresponding weak norms are
used. Yet another alternative class of data spaces has been considered
by Vargas and Vega in \cite{VV01}.

\medskip

Concerning the \eqref{eq:dnls} equation on the real line, it was shown
by the first author in \cite{G05}, that local well-posedness holds for
data in $\widehat{H}^{s}_{r}(\R)$, provided $s \geq \frac{1}{2}$ and $2
\geq r > 1$. This generalization of Takaoka's result almost reaches the
critical case, which is $(s,r)=(\frac{1}{2},1)$ in this setting. The
proof uses the gauge transform again and an appropriate variant of the
Fourier restriction norm method, which was developed in \cite{G03}.
Furthermore, it relies heavily on certain smoothing properties of the
Schr\"odinger group, expressed in terms of bi- and trilinear estimates
for free solutions.

\medskip

On the other hand it could be shown by the second author in
\cite{H06}, that Takaoka's result concerning the real line can be
carried over to the periodic case \emph{with the same lower bound} $s
\geq \frac{1}{2}$ on the Sobolev regularity. This is
remarkable, since there is a number of nonlinear Schr\"odinger and
Korteweg - de Vries type equations, which are -- due to a lack of smoothing properties
-- strictly worse behaved in the periodic setting than in the
continuous case. To prove the result concerning the one-dimensional
torus, the gauge transform had to be adjusted to the periodic case,
see Section 2 of \cite{H06}. The transformed equation is then treated
by the Fourier restriction norm method. Here, the $L^4$ Strichartz
estimate \cite{Zy74, Bo93a} turned out to be a central tool in
the derivation of the nonlinear estimates.

\medskip

Now it is natural to ask for a synthesis of the two last-mentioned
results, i. e., to consider the Cauchy problem \eqref{eq:dnls} with
$u_0$ in the following two parameter scale of data spaces.

\begin{definition}\label{def:sob}
  Let $s\in \R$, $1 \leq r \leq \infty$ and
  $\frac{1}{r}+\frac{1}{r'}=1$. Define $\widehat{H}^{s}_{r}(\T)$ as
  the completion of all trigonometric polynomials with respect to the
  norm
  \begin{equation}\label{eq:sob}
    \|f\|_{\widehat{H}^{s}_{r}}:=\|\widehat{J^sf}\|_{\ell^{r'}_{\xi}},
  \end{equation}
  where $J^s$ is the Bessel potential operator of order $-s$ given by
  $\widehat{J^s f}(\xi)=\lb\xi\rb^s \widehat{f}(\xi)$.
\end{definition}
\begin{remark}\label{rem:prop_sob}
  In contrast to the non-periodic case the continuous embedding
  $\widehat{H}^{s}_{q}(\T)\subset \widehat{H}^{s}_{r}(\T)$ holds true
  for any $1\leq r \leq q \leq \infty$. Moreover, we have
  $H^s_2(\T)=\widehat{H}^{s}_{2}(\T)$ and more generally
  $H^{s}_r(\T)\subset \widehat{H}^{s}_r(\T)$ for $1\leq r\leq 2$ by
  Hausdorff-Young, where $H^{s}_r(\T)$ denotes the Bessel potential
  space of all $u$ such that $J^s u \in L^{r}(\T)$. If $r=2$ we will
  usually omit the index $r$.
\end{remark}
The main result of this paper is local well-posedness of
\eqref{eq:dnls} in these data spaces in the parameter range $s\geq
\frac{1}{2}$ and $2>r>\frac{4}{3}$. More precisely, the following
theorem will be shown.
\begin{theorem}\label{thm:main}
  Let $\frac{4}{3}<q \leq r\leq 2$. For every
  \begin{equation*}
    u_0\in B_R:=\{ u_0 \in \widehat{H}^{\frac{1}{2}}_{r}(\T)\mid
    \|u_0\|_{\widehat{H}^{\frac{1}{2}}_{q}}< R\}
  \end{equation*}
  and $T \ls R^{-2q'-}$ there exists a solution $u \in
  C([-T,T],\widehat{H}^{\frac{1}{2}}_{r}(\T))$ of the Cauchy problem
  \eqref{eq:dnls}. This solution is the unique limit of smooth
  solutions and the map
  \begin{equation*}
    \left(B_R, \|\cdot\|_{\widehat{H}^{\frac{1}{2}}_{r}}\right)
    \longrightarrow C([-T,T],\widehat{H}^{\frac{1}{2}}_{r}(\T)):
    \quad u_0 \mapsto u
  \end{equation*}
  is continuous but not locally uniformly continuous. However, on
  subsets of $B_R$ with fixed $L^2$ norm it is locally Lipschitz
  continuous.
\end{theorem}
\begin{remark}\label{rem:main}
  \begin{enumerate}
  \item The uniqueness statement in the theorem above can be
    sharpened, see Remark \ref{rem:sharp_uni}.
  \item Our methods rely on the $L^2$ conservation law, but not on the complete integrability of
    \eqref{eq:dnls}, see \cite{KN78}, and also apply to nonlinearities with (say)
    additional polynomial terms of type $|u|^k u$
  \item Solution always means solution of the corresponding integral
    equation
    \begin{equation*}
      u(t)=e^{it\partial_x^2}u_0+
      i\int_0^t e^{i(t-t')\partial_x^2}\partial_x(|u|^2u)(t')dt'
      \; , \; t \in (-T,T).
    \end{equation*}
  \item In view of the counterexamples in \cite{H06}, Theorem 5.3, and
    \cite{H06diss}, Theorem 3.1.5, which are essentially of the same
    kind as the one already given in \cite{Tak99}, Proposition 3.3, we
    cannot expect any positive result for $s<\frac{1}{2}$.  Observe
    that the examples concerning the periodic case are monochromatic
    waves and so do not distinguish between an $\ell^2_{\xi}$- and an
    $\ell^{r'}_{\xi}$-norm. Concerning the second parameter $r$, we
    must leave open the question, whether or not there is local
    well-posedness for $r\leq \frac{4}{3}$. Nonetheless, we will show
    below that our result is optimal within the framework we use.
  \end{enumerate}
\end{remark}
Before we turn to details, let us point out, that in the periodic case
almost nothing is known about Cauchy problems with data in the
$\widehat{H}^{s}_r(\T)$ spaces. The only result we are aware of is due
to Christ \cite{C05a, C05b}, who considers the following modification
of the cubic nonlinear Schr\"odinger equation on the one-dimensional torus
\begin{equation}\label{eq:nls_christ}
  \tag{NLS*} i\partial_t u + \partial_x^2 u = \left(|u|^2- 2 \dashint_0^{2\pi} |u|^2 dx\right)u,
\end{equation}
with initial condition $u(0)=u_0 \in \widehat{H}^s_r(\T)$.  He shows
that for $s \geq 0$ and $r>1$ the solution map
\begin{equation*}
  S: H^{\sigma}(\T) \longrightarrow C([0,\infty),H^{\sigma}(\T))\cap C^1([0,\infty),H^{\sigma-2}(\T))
\end{equation*}
($\sigma$ sufficiently large) ``extends by continuity to a uniformly
continuous mapping from the ball centered at $0$ of [arbitrary] radius
$R$ in $\widehat{H}^{s}_r(\T)$ to
$C([0,\tau],\widehat{H}^{s}_r(\T))$'', where $\tau$ depends on $R$, see Theorem 1.1
in \cite{C05a}.\footnote{In \cite{C05a} the data spaces are denoted as
$\mathcal{H}^{s,p}(\T)$, which corresponds
to $\widehat{H}^{s}_{p'}(\T)$ in our terms.} This result is shown by a
\emph{new method of solution}, which is developed in \cite{C05a}, a
summary of this method is given in Section 1.5 of that paper. The
positive result in \cite{C05a} is supplemented in \cite{C05b} by a
statement of non-uniqueness: For $2>r>1$ there exists a non-vanishing
weak solution $u \in C([0,1],\widehat{H}^{0}_r(\T))$ of
\eqref{eq:nls_christ} with initial value $u_0\equiv 0$, see Theorem
2.3 in \cite{C05b}.\footnote{A precise definition of a weak solution is given in \cite{C05b}, Section 2.1.}

\medskip

Using the function spaces $X^{s,b}_r$, defined by the norms
\begin{equation*}
  \|u\|_{X^{s,b}_r} =
  \|
  \langle\xi\rangle^s \langle\tau+\xi ^2\rangle^b\widehat{u}
  \|_{\ell^{r'}_{\xi}L^{r'}_{\tau}},
\end{equation*}
where $\frac{1}{r}+\frac{1}{r'}=1$ (cf. \cite{G03}, Section 2), we can
actually show local \emph{well-posedness} of the initial value problem
associated to \eqref{eq:nls_christ} with data $u_0 \in
\widehat{H}^0_r(\T)$, $2>r>1$, thus giving an alternative proof (based
on the contraction mapping principle) of Christ's result from
\cite{C05a}. The argument also provides uniqueness of the solution in
the restriction norm space based on $X^{0,b}_r$, see (2.37) and (2.38) in
\cite{G03}.
It was the starting point for our investigations
concerning the \eqref{eq:dnls} equation and exhibits already some of
the main arguments, so let us sketch this proof:

\medskip

We define the trilinear Operator $C_1$ by its partial Fourier
transform (in the space variable only)
\begin{equation*}
  \widehat{C_1(u_1,u_2,u_3)}(\xi)=(2\pi)^{-1}
  \sum_{\genfrac{}{}{0pt}{}{\xi=\xi_1+\xi_2+\xi_3}{\xi_1 \not= \xi, \xi_2\not=\xi}}
  \widehat{u}_1(\xi_1)\widehat{u}_2(\xi_2)
  \widehat{\overline{u}}_3(\xi_3),
\end{equation*}
so that the (partial) Fourier transform of the nonlinearity in
\eqref{eq:nls_christ} becomes
\begin{equation*}
  \widehat{C_1(u,u,u)}(\xi) - (2\pi)^{-1}\widehat{u}^2(\xi)\widehat{\overline{u}}(-\xi).
\end{equation*}
By Theorem 2.3 from \cite{G03} it is sufficient to estimate the latter
appropriately in $X^{0,b}_r$-norms. Here, the second contribution turns
out to be harmless, cf. the end of the proof of Theorem
\ref{thm:tri_x1} below. So matters essentially reduce to show the
following estimate:

\begin{prop}\label{prop:nls_christ}
  Let $r>1$, $\varepsilon>0$ and $b>\frac{1}{r}$. Then
  \begin{equation*}
    \|C_1(u_1,u_2,u_3)\|_{X^{0,-\varepsilon}_r} \ls \prod_{i=1}^3 \|u_i\|_{X^{0,b}_r}.
  \end{equation*}
\end{prop}

\begin{proof}
  Choosing $f_i \in \ell ^{r'}_{\xi}L^{r'}_{\tau}$ such that
  $\|f_i\|_{\ell ^{r'}_{\xi}L^{r'}_{\tau}}=\|u_i\|_{X^{0,b}_r}$ the
  above estimate can be rewritten as
  \begin{equation}\label{eq:nls_tri}
    \left
      \|\lb\sigma_0\rb^{-\varepsilon}
      \sum_{\genfrac{}{}{0pt}{}{\xi=\xi_1+\xi_2+\xi_3}{\xi_1 \not= \xi, \xi_2\not=\xi}}
      \int\limits_{\tau=\tau_1+\tau_2+\tau_3} \prod_{i=1}^3 \frac{f_i(\xi_i,\tau_i)}{\lb\sigma_i\rb^b}
      \, d\tau_1d\tau_2
    \right\|_{\ell ^{r'}_{\xi}L^{r'}_{\tau}}
    \ls \prod_{i=1}^3\|f_i\|_{\ell ^{r'}_{\xi}L^{r'}_{\tau}},
  \end{equation}
  where $\sigma_0=\tau+\xi^2$, $\sigma_{i}=\tau_{i}+\xi_{i}^2$
  ($i=1,2$) and $\sigma_3=\tau_3-\xi_3^2$.  By H\"older's inequality
  and Fubini's theorem \eqref{eq:nls_tri} can be deduced from
  \begin{equation}\label{eq:sup_nls_christ}
    \sup_{\xi,\tau}\, \lb\sigma_0\rb^{-r\varepsilon}
    \sum_{\genfrac{}{}{0pt}{}{\xi=\xi_1+\xi_2+\xi_3}{\xi_1 \not= \xi, \xi_2\not=\xi}}
    \int\limits_{\tau=\tau_1+\tau_2+\tau_3} \prod_{i=1}^3\lb\sigma_i\rb^{-rb} \, d\tau_1d\tau_2 < \infty.
  \end{equation}
  Using the resonance relation
  \begin{equation}\label{eq:rr}
    2|\xi_1\xi_2+\xi\xi_3|=2|\xi-\xi_1||\xi-\xi_2|\le \sum_{i=0}^3\lb\sigma_i\rb\le \prod_{i=0}^3\lb\sigma_i\rb
  \end{equation}
  the left hand side of \eqref{eq:sup_nls_christ} is bounded by
  \begin{align*}
    & \sum_{\genfrac{}{}{0pt}{}{\xi=\xi_1+\xi_2+\xi_3}{\xi_1 \not=
        \xi, \xi_2\not=\xi}}
    \lb\xi-\xi_1\rb^{0-}\lb\xi-\xi_2\rb^{0-}\int\limits_{\tau=\tau_1+\tau_2+\tau_3}
    d\tau_1d\tau_2\prod_{i=1}^3\lb\sigma_i\rb^{-1-}\\
    \ls & \sum_{\genfrac{}{}{0pt}{}{\xi=\xi_1+\xi_2+\xi_3}{\xi_1 \not=
        \xi, \xi_2\not=\xi}} \lb\xi-\xi_1\rb^{0-}\lb\xi-\xi_2\rb^{0-}
    \lb \tau + \xi^2 -2(\xi-\xi_1)(\xi-\xi_2)\rb^{-1-},
  \end{align*}
  where in the last step we have used Lemma \ref{lem:fam2.12} twice.
  Setting $\Z^\ast=\Z\setminus\{0\}$, $n_{i}=\xi-\xi_{i}$ for $i=1,2$ and $r=n_1 n_2$ the last sum
  can be rewritten as
  \begin{equation*}
    \sum_{r\in \Z^{\ast}}\lb \tau + \xi^2 -2r\rb^{-1-}\lb r \rb^{0-}
    \sum_{\genfrac{}{}{0pt}{}{n_1,n_2 \in \Z^{\ast} }{r=n_1 n_2}} 1
  \end{equation*}
  which is bounded by a constant independent of $\xi$ and $\tau$,
  since the number of divisors of $r\in \N$ can be estimated by
  $c_{\varepsilon}r^{\varepsilon}$ for any positive ${\varepsilon}$.
\end{proof}

Three aspects of the preceding are worth to be emphasized in view of
our investigations
here.
\subsection*{Necessity of cancellations and correction terms:}
For $r<2$ the above argument breaks down completely without the
restrictions $\xi \neq \xi_1$ and $\xi \neq \xi_2$ in the sum over the
Fourier coefficients. As was pointed out already by Christ, this
cancellation comes from the correction term
$2\dashint_0^{2\pi}|u|^2dx u$ subtracted in the nonlinearity, and for any other
coefficient in front of this term one cannot obtain continuous
dependence (see \cite{C05a}, last sentence of Section 1.3 and the
remark before (2.6)).  A very similar cancellation turns out to be
fundamental in our analysis of the \eqref{eq:dnls} equation, but here
\emph{the corresponding correction term comes from the gauge
transform} in its periodic variant, which is discussed in Section
\ref{sect:gauge}, Remark \ref{rem:cancellation} below. In fact, the main contribution to the cubic
part of the transformed equation is given by $T^*(u,u,u)$, where
\begin{equation*}
  \widehat{T^*(u_1,u_2,u_3)}(\xi)=(2\pi)^{-1}
  \sum_{\genfrac{}{}{0pt}{}{\xi=\xi_1+\xi_2+\xi_3}{\xi_1 \not= \xi, \xi_2\not=\xi}}
  \widehat{u}_1(\xi_1)\widehat{u}_2(\xi_2)
  i \xi_3 \widehat{\overline{u}}_3(\xi_3).
\end{equation*}
Again, our argument would not work without the
restrictions $\xi\neq \xi_1$, $\xi\neq \xi_2$.

\subsection*{Modification of the norms:}
If we try to estimate the term $T^*(u_1,u_2,u_3)$ in an
$X^{s,b}_r$-norm in a similar manner as in the proof of Proposition
\ref{prop:nls_christ}, we have to get control over a whole derivative,
that is, on Fourier side, over the factor $\xi_3$. The complete
absence of smoothing effects (gaining derivatives) in the periodic
case forces us to get this control from the resonance relation
\eqref{eq:rr} only, which is the same for \eqref{eq:dnls} as
for \eqref{eq:nls_christ}. This means, that we have to choose the
$b$-parameters equal to $-\frac{1}{2}$ on the left and to
$+\frac{1}{2}$ on the right hand side of the estimate. Now, the
necessity to cancel the $\xi_3$-factor and the resonance relation lead
to the consideration of eight cases -- some of them being symmetric --
depending on which of the $\sigma$'s is maximal and on whether or not
$|\xi \xi_3| \ls |\xi_1 \xi_2|$, see the table in the proof of Theorem
\ref{thm:tri_x1} below. Picking out the (relatively harmless) subcase,
where $|\xi \xi_3| \ls |\xi_1 \xi_2|$ and $\sigma_0$ is maximal, so
that $\prod_{i=1}^3\lb\sigma_i\rb^{\frac{1}{6}} \le
\lb\sigma_0\rb^{\frac{1}{2}}$, we are in the situation of the above
proof, with a half derivative on each factor (as desired) but with a
$b$-parameter on the right of at most $\frac{2}{3}$, which means that
we end up with the non-optimal restriction $r>\frac{3}{2}$. This leads
us to \emph{introduce a fourth parameter} $p$ in the
$X^{s,b}_r$-norms, which is the H\"older exponent concerning the
$\tau$-integration and may differ from $r$, see Definition \ref{def:x}
below. In our application here we choose $p=2$, thus going back to
some extent to the meanwhile classical $X^{s,b}$-spaces.

\subsection*{Number of divisor estimate:}
The number of divisor argument at the end of the proof of Proposition
\ref{prop:nls_christ} has been used already in Christ's work and can
be seen as a substitute for Bourgain's $L^6$ Strichartz estimate for
the periodic case, which itself was shown by the aid of this argument,
see  Proposition 2.36 in \cite{Bo93a}. We will need a refined version
thereof, which is shown by elementary geometric considerations in
Section \ref{sect:nodc}. Here, we use arguments similar to those
of De Silva, Pavlovic, Staffilani, and Tzirakis \cite{DPST06},
Section 4.

\medskip

Concerning the organization of the paper the following should be
added: In Section \ref{sect:spaces} we introduce the relevant function
spaces and state all the nonlinear estimates needed as well as a
sharpness result. The crucial trilinear estimates and a counterexample
are derived in Section \ref{sect:tri}, which is very much in the spirit
of \cite{KPV96}. Section \ref{sect:quinti} deals with the quintilinear
estimate.  In both cases we have made some effort to extract the correct
lifespan from the nonlinear estimates and to obtain persistence of
higher regularity. By this we mean that the lifespan of a solution
with $\widehat{H}^{\frac{1}{2}}_r(\T)$-data only depends on the
smaller $\widehat{H}^{\frac{1}{2}}_q(\T)$-norm of the initial value,
where $2 \ge r > q > \frac{4}{3}.$\footnote{We refer to \cite{Se01},
  Theorem 2, part V, for the corresponding notion, if data in the
  $H^s$-scale are considered. In the $H^s$-case this property usually
  is a simple consequence of the convolution constraint - see again
  \cite{Se01}, Remark 2 below Theorem 3. We cannot see that a similar
  argument should work in our setting.} Finally, in Section
\ref{sect:proof_wp}, the contraction mapping principle is invoked to
prove local well-posedness for the transformed equation
\eqref{eq:gauge_dnls}, see Theorem \ref{thm:wp_gauge_dnls}. Our main
result, Theorem \ref{thm:main}, is then a consequence of Lemma
\ref{lemma:gauge_prop} on the gauge transform.

We close this section by fixing some notational conventions.
\begin{itemize}
\item The Fourier transform with respect to the space variable
  (periodic)
  \begin{align*}
    \mathcal{F}_x f(\xi) =\widehat{f}(\xi)
    =\frac{1}{\sqrt{2\pi}}\int_0^{2\pi} f(x) e^{-ix\xi} dx \quad (f
    \in L^1(\T))
  \end{align*}
\item The Fourier transform with respect to the time variable
  (nonperiodic)
  \begin{equation*}
    \mathcal{F}_t f(\tau)
    =\frac{1}{\sqrt{2\pi}}\int_\R f(t) e^{-it\tau} dt\quad (f \in L^1(\R))
  \end{equation*}
\item The Fourier transform with respect to time and space variables
  $\mathcal{F}=\mathcal{F}_t\mathcal{F}_x$
\item For the mean value integral we write
  $$\dashint_0^{2\pi} f(x)  dx=\frac{1}{2\pi} \int_0^{2\pi} f(x) dx
  \quad (f \in L^1(\T))$$
\item Let $a\in \R$. The expressions $a\pm$ denote numbers $a\pm\eps$
  for an arbitrarily small $\eps>0$
\item For a given set of parameters (typically a subset of
  $\eps,\delta,\nu,p,q,r,s$) the statement $A \ls B$ means that there
  exists a constant $C>0$ which depends only on these parameters such
  that $A\leq C B$. This is equivalent to $B\gs A$.  We may write
  $A\ll B$ if it is possible to choose $0<C<\frac{1}{4}$.
\item For all parameters $1\leq p\leq \infty$ the number $1\leq p'\leq
  \infty$ is defined to be the dual parameter satisfying
  $\frac{1}{p}+\frac{1}{p'}=1$.
\end{itemize}

\subsection*{Acknowledgement}
The second author gratefully acknowledges partial support from the
German Research Foundation (DFG), grant KO 1307/5-3.

\section{Function spaces and main estimates}\label{sect:spaces}
\noindent
Let $\test$ be the linear space of all $C^\infty$-functions $f:\R^2
\to \C$ such that
\begin{equation*}
  f(t,x)=f(t,x+2\pi)\; , \,
  \sup_{(t,x) \in \R^2} |t^\alpha \partial^\beta_t\partial^\gamma_x
  f(t,x)|<\infty \;, \alpha,\beta,\gamma \in \N_0
\end{equation*}
\begin{definition}\label{def:x} Let $s,b \in \R$, $1 \leq r,p
  \leq \infty$ and
  $\frac{1}{r}+\frac{1}{r'}=1=\frac{1}{p}+\frac{1}{p'}$.  Define the
  space $X^{s,b}_{r,p}$ as the completion of $\test$ with respect to
  the norm
  \begin{equation}\label{eq:x_norm}
    \|u\|_{X^{s,b}_{r,p}}=\|\lb\tau+\xi^2\rb^{b}\lb \xi \rb^{s}
    \mathcal{F}u\|_{\ell^{r'}_{\xi}L^{p'}_{\tau}}
  \end{equation}
\end{definition}
In the case where $r=p=2$ we write $X^{s,b}_{r,p}=X^{s,b}$ as usual.
\begin{lemma}\label{lem:cont_emb} Let $s,b_1,b_2\in \R$, $1\leq r\leq \infty$ and $b_1>b_2+\frac{1}{2}$. 
  The following embeddings are continuous:
  \begin{align}
    \label{eq:x_emb} X^{s,b_1}_{r,2} &\subset X^{s,b_2}_{r,\infty}\\
    \label{eq:pers} X^{s,0}_{r,\infty} &\subset
    C(\R,\widehat{H}^{s}_{r}(\T))
  \end{align}
\end{lemma}
\begin{proof}
  The first embedding is proved by the Cauchy-Schwarz
  inequality with respect to the $L^1_\tau$ norm. The second embedding follows from
  $L^\infty(\R) \subset \mathcal{F}^{-1}_t L^1(\R)$.
\end{proof}
\begin{definition}\label{def:res_space}
  Let $s\in \R$ and $1\leq r \leq \infty$. We define
  \begin{equation*}
    Z^s_r:=X^{s,\frac{1}{2}}_{r,2} \cap X^{s,0}_{r,\infty}
  \end{equation*}
  and for $0<T\leq 1$ the restriction space $Z^{s}_{r}(T)$ of all
  $v=w\mid_{[-T,T]}$ for some $w\in Z^s_r$ with norm
  \begin{equation*}
    \|v\|_{Z^{\frac{1}{2}}_{r}(T)}
    :=\inf\{\|w\|_{Z^{\frac{1}{2}}_r}|w \in Z^{\frac{1}{2}}_r: w\mid_{[-T,T]}=v\}
  \end{equation*}
\end{definition}
A main ingredient for the proof of Theorem \ref{thm:main} is an
estimate on the trilinear operator (suppressing the $t$ dependence)
\begin{equation}\label{eq:mod_tri}
  \begin{split}
    \widehat{T(u_1,u_2,u_3)}(\xi)=&(2\pi)^{-1}
    \sum_{\genfrac{}{}{0pt}{}{\xi=\xi_1+\xi_2+\xi_3} {\xi_1 \not= \xi,
        \xi_2\not=\xi}} \widehat{u}_1(\xi_1)\widehat{u}_2(\xi_2)i
    \xi_3
    \widehat{\overline{u}}_3(\xi_3)\\
    &+(2\pi)^{-1} \widehat{u}_1(\xi)\widehat{u}_2(\xi)i
    \xi\widehat{\overline{u}}_3(-\xi)
  \end{split}
\end{equation}
\begin{theorem}\label{thm:tri_x1}
  Let $\frac{4}{3}<q\leq r \leq 2$ and $0 \leq \delta<\frac{1}{q'}$.
  Then,
  \begin{equation}\label{eq:tri_x1}
    \|T(u_1,u_2,u_3)\|_{X^{\frac{1}{2},-\frac{1}{2}}_{r,2}} \ls T^\delta
    \|u_1\|_{X^{\frac{1}{2},\frac{1}{2}}_{q,2}}\|u_2\|_{X^{\frac{1}{2},\frac{1}{2}}_{q,2}}
    \|u_3\|_{X^{\frac{1}{2},\frac{1}{2}}_{r,2}}
  \end{equation}
  if $\supp(u_i) \subset \{(t,x) \mid |t| \leq T\}$, $0<T \leq 1$.
\end{theorem}
Additionally, we will need the following estimate on $T(u_1,u_2,u_3)$.
\begin{theorem}\label{thm:tri_x2}
  Let $\frac{4}{3}<q\leq r \leq 2$ and $0 \leq \delta<\frac{1}{q'}$.
  Then,
  \begin{equation}\label{eq:tri_x2}
    \|T(u_1,u_2,u_3)\|_{X^{\frac{1}{2},-1}_{r,\infty}} \ls T^\delta
    \|u_1\|_{X^{\frac{1}{2},\frac{1}{2}}_{q,2}}\|u_2\|_{X^{\frac{1}{2},\frac{1}{2}}_{q,2}}
    \|u_3\|_{X^{\frac{1}{2},\frac{1}{2}}_{r,2}}
  \end{equation}
  if $\supp(u_i) \subset \{(t,x) \mid |t| \leq T\}$, $0<T \leq 1$.
\end{theorem}
The above estimates are sharp with respect to the lower threshold on
$r$ within the full scale of spaces
$X^{\frac{1}{2},\frac{1}{2}}_{r,p}$.  Note that in particular the
estimates fail to hold in the endpoint case $r=\frac{4}{3}$.
\begin{remark}\label{rem:counter}
  For all $b\leq 0$, $1 \leq r \leq \frac{4}{3}$ and $1\leq p,q \leq
  \infty$ the estimate
  \begin{equation}\label{eq:counter}
    \|T(u_1,u_2,u_3)\|_{X^{\frac{1}{2},b}_{r,p}} \ls
    \prod_{i=1}^3 \|u_i\|_{X^{\frac{1}{2},\frac{1}{2}}_{r,q} }
  \end{equation}
  is false.
\end{remark}

We also consider the quintilinear expression defined as
\begin{equation}\label{eq:mod_quinti}
  \widehat{Q(u_1,\ldots,u_5)}(\xi)
  =(2\pi)^{-2}
  \sum_{\ast\ast}
  \widehat{u_1}(\xi_1) \widehat{\overline{u_2}}(\xi_2)
  \widehat{u_3}(\xi_3) \widehat{\overline{u_4}}(\xi_4) \widehat{u_5}(\xi_5)
\end{equation}
where we suppressed the $t$ dependence and $\ast\ast$ is shorthand for
summation over the subset of $\Z^5$ given by the restrictions
\begin{equation*}
  \xi=\xi_1+\ldots+\xi_5 \; ;\; \xi_1+\ldots+\xi_4\not=0\; ; \;
  \xi_1+\xi_2\not=0 \; ;\; \xi_3+\xi_4 \not=0
\end{equation*}
\begin{theorem}\label{thm:quinti}
  Let $\frac{4}{3}<q\leq r \leq 2$ and $b>\frac{1}{6}+\frac{1}{3q}$.
  Then,
  \begin{equation}\label{eq:quinti_gen}
    \|u_1\overline{u}_2u_3\overline{u}_4u_5\|_{X^{\frac{1}{2},-b}_{r,2}} \ls \sum_{k=1}^5 
    \|u_k\|_{X^{\frac{1}{2},b}_{r,2}} 
    \prod_{\genfrac{}{}{0pt}{}{1 \leq i \leq 5}{i \not= k}}
    \|u_i\|_{X^{\frac{1}{2},b}_{q,2}}
  \end{equation}
  Additionally assume that for $0<T \leq 1$ we have $\supp(u_i)
  \subset \{(t,x) \mid |t| \leq T\}$ and $0 \leq \delta<\frac{2}{q'}$.
  Then,
  \begin{equation}\label{eq:quinti}
    \|u_1\overline{u}_2u_3\overline{u}_4u_5\|_{X^{\frac{1}{2},-\frac{1}{2}}_{r,2}
      \cap 
      X^{\frac{1}{2},-1}_{r,\infty}} \ls T^\delta \sum_{k=1}^5 
    \|u_k\|_{X^{\frac{1}{2},\frac{1}{2}}_{r,2}} 
    \prod_{\genfrac{}{}{0pt}{}{1 \leq i \leq 5}{i \not= k}}
    \|u_i\|_{X^{\frac{1}{2},\frac{1}{2}}_{q,2}}
  \end{equation}
  and
  \begin{equation}\label{eq:quinti_a}
    \|Q(u_1,u_2,u_3,u_4,u_5)\|_{X^{\frac{1}{2},-\frac{1}{2}}_{r,2}
      \cap 
      X^{\frac{1}{2},-1}_{r,\infty}} \ls T^\delta \sum_{k=1}^5 
    \|u_k\|_{X^{\frac{1}{2},\frac{1}{2}}_{r,2}} 
    \prod_{\genfrac{}{}{0pt}{}{1 \leq i \leq 5}{i \not= k}}
    \|u_i\|_{X^{\frac{1}{2},\frac{1}{2}}_{q,2}}
  \end{equation}
\end{theorem}

\section{Number of divisor estimates and
  consequences}\label{sect:nodc}
\noindent
The next lemma contains estimates on the number of divisors of a given
natural number $r$.  Part \ref{it:num_div} is well-known (see
Hardy-Wright \cite{HW79}, Theorem 315).  The approach used to prove Part \ref{it:num_div_ref} of Lemma \ref{lem:divisors} is
motivated by \cite{DPST06}, Lemma 4.4.
\begin{lemma}\label{lem:divisors}
  \begin{enumerate}
  \item\label{it:num_div}  Let $\eps>0$. There exists $c_\eps>0$, such that for all $r \in \N$
    \begin{equation}\label{eq:num_div}
      \# \left\{ (n_1,n_2) \in \N^2 \mid n_1n_2=r\right\} \leq c_\eps r^\eps
    \end{equation}
  \item\label{it:num_div_ref} For all $r \in \N$
    \begin{equation}\label{eq:num_div_ref}
      \# \left\{ (n_1,n_2) \in \N^2 \mid n_1n_2=r, \quad 3|n_1-n_2|\leq
        r^{\frac{1}{6}}\right\} \leq 2
    \end{equation}
  \end{enumerate}
\end{lemma}
\begin{proof}[Proof of Part \ref{it:num_div_ref}]
  Let $r \in \N$.  Assume that there are three lattice points
  contained in the above set. Then, these points form a triangle of
  area $\mu \geq \frac{1}{2}$. This triangle is located
  \begin{enumerate}
  \item in the strip
    $$S=\left\{ (x_1,x_2) \in \R^2 \mid \sqrt{r}-\delta \leq x_1 \leq \sqrt{r}
      +\delta\right\}$$ where $\delta=\frac{1}{3}r^{\frac{1}{6}}$,
    because $|n_1-\sqrt{r}|\leq |n_1-n_2| \leq \delta$.
  \item below the line $L=\overline{P_1P_2}$ which connects the points
    $$
    P_1=\left(\sqrt{r}-\delta,\frac{r}{\sqrt{r}-\delta}\right) \quad,
    \quad P_2=\left(\sqrt{r}+\delta,\frac{r}{\sqrt{r}+\delta}\right)
    $$
  \item above the hyperbola
    $$
    H=\left\{(x_1,x_2) \in (0,\infty)^2 \mid x_1x_2=r \right\}
    $$
    because the function $x \mapsto \frac{r}{x}$ is convex for $x>0$.
  \end{enumerate}
  \begin{figure}[t]
    \begin{picture}(0,0)%
      \includegraphics{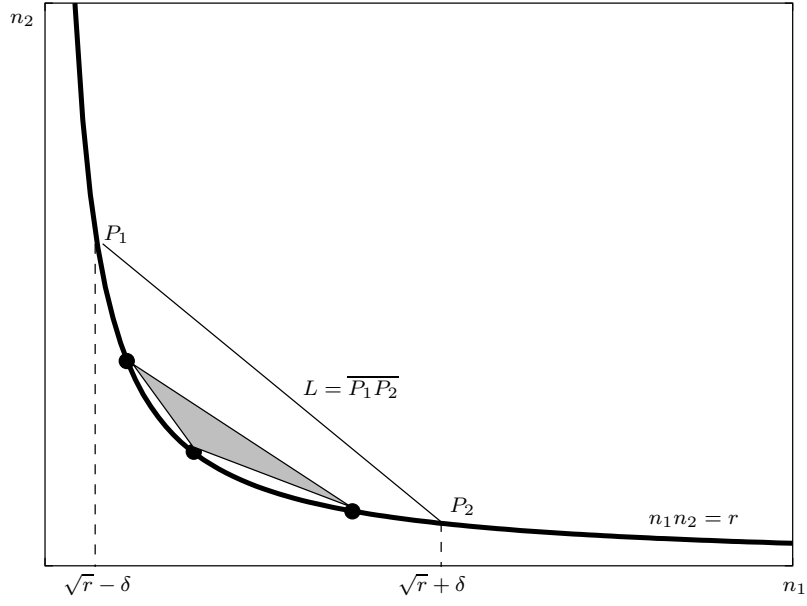}%
    \end{picture}%
    \setlength{\unitlength}{4144sp}%
    \begin{picture}(4761,3595)(976,-6772)
      \put(4600,-6500){\makebox(0,0)[lb]{\smash{\footnotesize$n_1n_2=r$}}}
      \put(5401,-6900){\makebox(0,0)[lb]{\smash{\footnotesize$n_1$}}}
      \put(3410,-6430){\makebox(0,0)[lb]{\smash{\footnotesize$P_2$}}}
      \put(2530,-5731){\makebox(0,0)[lb]{\smash{\footnotesize$L=\overline{P_1P_2}$}}}
      \put(1340,-4800){\makebox(0,0)[lb]{\smash{\footnotesize$P_1$}}}
      \put(780,-3500){\makebox(0,0)[lb]{\smash{\footnotesize$n_2$}}}
      \put(1100,-6900){\makebox(0,0)[lb]{\smash{\footnotesize$\sqrt{r}-\delta$}}}
      \put(3100,-6900){\makebox(0,0)[lb]{\smash{\footnotesize$\sqrt{r}+\delta$}}}
    \end{picture}
    \caption{\label{fig:tri} The notional triangle from the proof of
      Lemma \ref{lem:divisors}}
  \end{figure}
  Now, the area $\mu$ of the triangle is bounded from above by the
  area of the region in the strip $S$ below $L$ and above $H$, hence
  \begin{align*}
    \mu \leq &
    r\delta\left(\frac{1}{\sqrt{r}-\delta}+\frac{1}{\sqrt{r}+\delta}\right)-\int_{\sqrt{r}
      -\delta}^{\sqrt{r}+\delta} \frac{r}{x} \, dx\\
    =&
    r\left(\frac{2\sqrt{r}\delta}{r-\delta^2}-\ln\left(1+\frac{2\delta}{\sqrt{r}-\delta}
      \right)\right)\\
    \leq & r\left(\frac{2\sqrt{r}\delta}{r-\delta^2}-
      \frac{2\delta}{\sqrt{r}-\delta}+\frac{2\delta^2}{(\sqrt{r}-\delta)^2}\right)
    = \frac{4r\delta^3}{(\sqrt{r}-\delta)^2(\sqrt{r}+\delta)}\\
    \leq & \frac{4\sqrt{r}\delta^3}{(\sqrt{r}-\delta)^2} \leq
    \frac{4}{27} \frac{r}{(\sqrt{r}-\delta)^2} \leq \frac{1}{3}
  \end{align*}
  which contradicts $\mu\geq \frac{1}{2}$.
\end{proof}
Now, we use Lemma \ref{lem:divisors} to prove
\begin{corollary}\label{cor:sums}
  Fix $\eps>0$.
  \begin{enumerate}
  \item There exists $C_{\eps}>0$ such that for all $\xi \in \Z$ and
    $a \in \R$
    \begin{equation}\label{eq:sum01}
      \sum_{\genfrac{}{}{0pt}{}{\xi_1,\xi_2 \in \Z}{\xi_1,\xi_2 \not=
          \xi}}\lb \xi-\xi_1\rb^{-\eps}
      \lb \xi-\xi_2\rb^{-\eps}\lb a+2(\xi-\xi_1)(\xi-\xi_2)\rb^{-1-\eps}
      \leq C_{\eps}
    \end{equation}
  \item There exists $C_{\eps}>0$ such that for all $\xi_1 \in \Z$ and
    $a \in \R$
    \begin{equation}\label{eq:sum02}
      \sum_{\genfrac{}{}{0pt}{}{\xi,\xi_2 \in \Z}{\xi_1,\xi_2 \not=
          \xi}}\lb \xi-\xi_1\rb^{-\eps}
      \lb \xi-\xi_2\rb^{-\eps}\lb a+2(\xi-\xi_1)(\xi-\xi_2)\rb^{-1-\eps}
      \leq C_{\eps}
    \end{equation}
  \item There exists $C_{\eps}>0$ such that for all $\xi \in \Z$ and
    $a \in \R$
    \begin{equation}\label{eq:sum1}
      \sum_{\genfrac{}{}{0pt}{}{\xi_1,\xi_2 \in \Z}{\xi_1,\xi_2 \not= \xi}} \lb
      \xi_1\rb^{-\eps}\lb \xi_2\rb^{-\eps}\lb a+2(\xi-\xi_1)(\xi-\xi_2)\rb^{-1-\eps}
      \leq C_{\eps}
    \end{equation}
  \item There exists $C_{\eps}>0$ such that for all $\xi_1 \in \Z$ and
    $a \in \R$
    \begin{equation}\label{eq:sum2}
      \lb \xi_1\rb^{-\eps}\sum_{\genfrac{}{}{0pt}{}{\xi,\xi_2 \in \Z}{\xi_1,\xi_2 \not=\xi}}
      \lb \xi_2\rb^{-\eps}\lb a+2(\xi-\xi_1)(\xi-\xi_2)\rb^{-1-\eps}
      \leq C_{\eps}
    \end{equation}
  \end{enumerate}
\end{corollary}
\begin{proof}
  The first and second part follow from the standard number of
  divisors estimate \eqref{eq:num_div} as follows: By the change of variables
  $n_1=\xi-\xi_1$, $n_2=\xi-\xi_2$ the sums in \eqref{eq:sum01} and
  \eqref{eq:sum02} are equal to
  \begin{equation*}
    \sum_{n_1,n_2 \in \Z^\ast}\lb n_1\rb^{-\eps}
    \lb n_2\rb^{-\eps}\lb a+2n_1n_2\rb^{-1-\eps}
  \end{equation*}
  This can be written as
  \begin{align*}
    &\sum_{r\in \Z^\ast} \sum_{\genfrac{}{}{0pt}{}{n_1,n_2 \in
        \Z^\ast}{n_1n_2=r}}
    \lb n_1\rb^{-\eps}\lb n_2\rb^{-\eps}\lb a+2n_1n_2\rb^{-1-\eps}\\
    \leq &\sum_{r\in \Z^\ast}\lb a+2 r\rb^{-1-\eps}
    \lb r\rb^{-2\eps} \# \left\{ (n_1,n_2) \in (\Z^\ast)^2 \mid n_1n_2=r\right\}\\
    \leq & c_\eps \sum_{r\in \Z^\ast}\lb a+2 r\rb^{-1-\eps}
  \end{align*}
  for some $c_\eps>1$. We write $a=2b+\delta$, $b \in \Z$, $\delta \in
  [0,2)$ and
  \begin{equation*}\sum_{r \in \Z^\ast} \lb a+2r\rb^{-1-\eps} \leq
    \sum_{r \in \Z}\lb
    r+\delta\rb^{-1-\eps} \leq 
    3+2\sum_{r \in \N}\lb r\rb^{-1-\eps}=:s_\eps
  \end{equation*}
  Now, the estimates \eqref{eq:sum01} and \eqref{eq:sum02} hold with
  $C_\eps:=s_\eps c_\eps$.  In order to show formula \eqref{eq:sum1}
  of the third part we use the same change of variables as above and
  obtain
  \begin{equation*}
    \sum_{r \in \Z^\ast} \lb a+2r\rb^{-1-\eps}
    \sum_{\genfrac{}{}{0pt}{}{n_1,n_2 \in \Z^\ast}{n_1 n_2 =r}} \lb
    \xi-n_1\rb^{-\eps} \lb \xi-n_2\rb^{-\eps}
  \end{equation*}
  Let $M(r)=\{(n_1,n_2) \in (\Z^\ast)^2 \mid n_1n_2=r\}$. Now, we
  split the inner sum into two parts. Let
  \begin{equation*}
    M_1(r)= \left\{(n_1,n_2) \in M(r)\mid  6|\xi-n_1| \geq |r|^{\frac{1}{6}}
      \text{ or }
      6|\xi-n_2| \geq |r|^{\frac{1}{6}} \right\}
  \end{equation*}
  and
  \begin{equation*}
    M_2(r) =\left\{(n_1,n_2) \in M(r)
      \mid 6|\xi-n_1| \leq |r|^{\frac{1}{6}}
      \text{ and } 6|\xi-n_2| \leq |r|^{\frac{1}{6}} \right\}
  \end{equation*}
  Obviously we have $M(r)=M_1(r)\cup M_2(r)$.  By Part
  \ref{it:num_div} of Lemma \ref{lem:divisors} there exists $c_\eps>1$
  such that
  $$
  \# M_1(r) \leq 2 \, \#\left\{(n_1,n_2) \in \N^2 \mid n_1n_2=r
  \right\} \leq c_\eps |r|^{\frac{\eps}{6}}
  $$
  and it follows
  $$
  \sum_{(n_1,n_2) \in M_1(r)} \lb \xi-n_1\rb^{-\eps} \lb
  \xi-n_2\rb^{-\eps} \leq 6^\eps|r|^{-\frac{\eps}{6}} \# M_1(r) \leq
  6^\eps c_\eps
  $$
  For $(n_1,n_2) \in M_2(r)$ it holds that $3|n_1-n_2|\leq
  |r|^{\frac{1}{6}}$. An application of Part \ref{it:num_div_ref} of
  Lemma \ref{lem:divisors} shows
  $$
  \sum_{(n_1,n_2) \in M_2(r)} \lb \xi-n_1\rb^{-\eps} \lb
  \xi-n_2\rb^{-\eps} \leq \#M_2(r)\leq 4
  $$
  Therefore, we see that
  $$
  \sum_{\genfrac{}{}{0pt}{}{\xi_1,\xi_2 \in \Z}{\xi_1,\xi_2 \not=
      \xi}} \lb \xi_1\rb^{-\eps}\lb \xi_2\rb^{-\eps}\lb
  a+2(\xi-\xi_1)(\xi-\xi_2)\rb^{-1-\eps} \leq (6^\eps c_\eps+4)
  \sum_{r \in \Z^\ast} \lb a+2r\rb^{-1-\eps}
  $$
  and the third part is proved with constant $C_\eps=s_\eps(6^\eps
  c_\eps+4)$.  Concerning the fourth part we proceed similarly. After
  changing variables
  \begin{align*}
    &\lb \xi_1\rb^{-\eps}\sum_{\genfrac{}{}{0pt}{}{\xi,\xi_2 \in
        \Z}{\xi_1,\xi_2 \not=\xi}}
    \lb \xi_2\rb^{-\eps}\lb a+2(\xi-\xi_1)(\xi-\xi_2)\rb^{-1-\eps}\\
    =& \lb \xi_1\rb^{-\eps}\sum_{r \in \Z^\ast} \lb a+2r\rb^{-1-\eps}
    \sum_{\genfrac{}{}{0pt}{}{n_1,n_2 \in \Z^\ast}{n_1 n_2 =r}}\lb
    \xi_1+n_1-n_2\rb^{-\eps}
  \end{align*}
  we consider two subregions of summation. In the case where
  $|r|^{\frac{1}{6}} \leq 6|\xi_1|$ or $|r|^{\frac{1}{6}}\leq
  6|\xi_1+n_1-n_2|$ we apply estimate \eqref{eq:num_div} from the
  first part of Lemma \ref{lem:divisors}, while in the remaining case
  it holds $3|n_1-n_2|\leq |r|^{\frac{1}{6}}$ and we utilize estimate
  \eqref{eq:num_div_ref} from the second part of Lemma
  \ref{lem:divisors}.
\end{proof}
\section{The proof of the trilinear estimates}\label{sect:tri}
\noindent
In this section we prove Theorem \ref{thm:tri_x1} and Theorem
\ref{thm:tri_x2}.  We will frequently use the following well-known
(see e.g.  \cite{GTV97}, Lemma 4.2) tool:
\begin{lemma}\label{lem:fam2.12}
  Let $0\leq \alpha \leq \beta$ such that $\alpha+\beta>1$ and
  $\eps>0$.  Then,
  \begin{equation*}
    \int_\R \lb s-a\rb^{-\alpha}\lb s-b\rb^{-\beta} ds \ls \lb
    a-b\rb^{-\gamma}\; , \; \gamma=
    \begin{cases}
      \alpha+\beta-1&, \beta <1\\
      \alpha-\eps &,\beta=1\\
      \alpha &,\beta >1
    \end{cases}
  \end{equation*}
\end{lemma}
We write $T=T^\ast+T^{\ast\ast}$, where
\begin{align*}
  \widehat{T^{\ast}(u_1,u_2,u_3)}(\xi)&=(2\pi)^{-1}
  \sum_{\genfrac{}{}{0pt}{}{\xi=\xi_1+\xi_2+\xi_3} {\xi_1 \not= \xi,
      \xi_2\not=\xi}} \widehat{u}_1(\xi_1)\widehat{u}_2(\xi_2)i \xi_3
  \widehat{\overline{u}}_3(\xi_3)\\
  \widehat{T^{\ast\ast}(u_1,u_2,u_3)}(\xi)&=(2\pi)^{-1}
  \widehat{u}_1(\xi)\widehat{u}_2(\xi)i
  \xi\widehat{\overline{u}}_3(-\xi)
\end{align*}
\begin{proof}[Proof of Theorem \ref{thm:tri_x1}]
  To fix notation, let $\sigma_0=\tau+\xi^2$,
  $\sigma_j=\tau_j+\xi_j^2$, $j=1,2$ and $\sigma_3=\tau_3-\xi_3^2$.
  Throughout the proof the quantities $\xi_3, \tau_3$ are defined as
  $\xi_3=\xi-\xi_1-\xi_2$ and $\tau_3=\tau-\tau_1-\tau_2$,
  respectively. Let us denote $\mu=(\tau,\xi),\mu_i=(\tau_i,\xi_i)$,
  $i=1,2,3$ for brevity. By the definition of the norms we may assume
  that $\widehat{u}_j\geq 0$. Then,
  \begin{equation*}
    \|T(u_1,u_2,u_3)\|_{X^{\frac{1}{2},-\frac{1}{2}}_{r,2}} \leq
    \|T^\ast(u_1,u_2,u_3)\|_{X^{\frac{1}{2},-\frac{1}{2}}_{r,2}}+
    \|T^{\ast\ast}(u_1,u_2,u_3)\|_{X^{\frac{1}{2},-\frac{1}{2}}_{r,2}}
  \end{equation*}
  and we consider the contribution from $T^{\ast}$ first: Let $m$ be
  given by
  \begin{equation*}
    m(\mu,\mu_1,\mu_2)
    =\frac{\lb\xi\rb^{\frac{1}{2}} i\xi_3}{
      \prod_{j=1}^3\lb\xi_j\rb^{\frac{1}{2}}
      \prod_{j=0}^3\lb\sigma_j\rb^{\frac{1}{2}}}
  \end{equation*}
  Estimate \eqref{eq:tri_x1} for the $T^\ast$ contribution is
  equivalent to
  \begin{equation}\label{eq:tri_mult_x}
    \begin{split}
      &\left\|\sum_{\genfrac{}{}{0pt}{}{\xi_1,\xi_2 \in \Z} {\xi_1
            \not= \xi, \xi_2\not=\xi}}\int
        m(\mu,\mu_1,\mu_2)f_1(\mu_1)f_2(\mu_2)
        f_3(\mu_3) d\tau_1d\tau_2 \right\|_{{\ell}^{r'}_{\xi}L^2_{\tau}}\\
      \ls & T^\delta
      \|f_1\|_{{\ell}^{q'}_{\xi}L^2_{\tau}}\|f_2\|_{{\ell}^{q'}_{\xi}L^2_{\tau}}
      \|f_3\|_{{\ell}^{r'}_{\xi}L^2_{\tau}}
    \end{split}
  \end{equation}
  where we may assume $f_3(\tau_3,0)=0$.  The resonance relation
  \begin{equation}\label{eq:res_id}
    \sigma_0-\sigma_1-\sigma_2-\sigma_3=2(\xi-\xi_1)(\xi-\xi_2)=2(\xi_1\xi_2+\xi\xi_3)
  \end{equation}
  holds true, cp. \cite{Tak99,G05,H06}.  Let us first consider the
  subregion where $\lb\xi_1\rb \lb \xi_2\rb \ll \lb\xi\rb\lb\xi_3\rb$.
  Then,
  \begin{equation}\label{eq:res1}
    \lb \xi\rb^\frac{1}{2}\lb \xi_3\rb^\frac{1}{2}\ls \sum_{k=0}^3\lb
    \sigma_k \rb^{\frac{1}{2}}
  \end{equation}
  and in this subregion we control $|m|$ by the sum of all
  \begin{equation*}
    m_{k,1}(\mu,\mu_1,\mu_2)
    =\frac{1}{
      \lb\xi_1\rb^{\frac{1}{2}}\lb\xi_2\rb^{\frac{1}{2}}
      \prod_{j=0,j\not=k}^3\lb\sigma_j\rb^{\frac{1}{2}}}
  \end{equation*}
  for $k=0,\ldots,3$.  Secondly, in the subregion where
  $\lb\xi\rb\lb\xi_3\rb \ls \lb\xi_1\rb \lb \xi_2\rb$ (note that
  $\xi_1\not=\xi$, $\xi_2\not=\xi$ within the domain of summation) it
  holds
  \begin{equation}\label{eq:res2}
    \lb \xi-\xi_1\rb^\frac{1}{2}\lb \xi-\xi_2\rb^\frac{1}{2}
    \ls \sum_{k=0}^3\lb \sigma_k \rb^{\frac{1}{2}}
  \end{equation}
  and in this subregion we control $|m|$ by the sum of all
  \begin{equation*}
    m_{k,2}(\mu,\mu_1,\mu_2)
    =\frac{1}{
      \lb\xi-\xi_1\rb^{\frac{1}{2}}\lb\xi-\xi_2\rb^{\frac{1}{2}}
      \prod_{j=0,j\not=k}^3\lb\sigma_j\rb^{\frac{1}{2}}}
  \end{equation*}
  for $k=0,\ldots,3$.  According to these multipliers we subdivide the
  proof into the following cases (with a preview of the lower bound on
  $q$ for each subcase obtained by our arguments below):
  \begin{center}
    \begin{tabular}{|c|c|c|c|c|}\hline
      & $\sigma_0=\max$ & $\sigma_1=\max$&$\sigma_2=\max$ &$\sigma_3=\max$ \\\hline
      $\lb\xi_1\rb \lb \xi_2\rb \ll \lb\xi\rb\lb\xi_3\rb$& Case 0.1:
      &Case 1.1:
      &Case 2.1: & Case 3.1:\\
      &$q>1$&$q>4/3$&$q>4/3$&$q>4/3$
      \\\hline
      $\lb\xi\rb\lb\xi_3\rb \ls \lb\xi_1\rb \lb \xi_2\rb$  & Case 0.2:
      &Case 1.2:
      &Case 2.2: & Case 3.2:\\
      &$q>1$&$q>4/3$&$q>4/3$&$q>4/3$
      \\\hline
    \end{tabular}
  \end{center}
  For technical reasons, we will prove the slightly stronger estimates
  \begin{equation}\label{eq:tri_mult_x_nu}
    \begin{split}
      & \left\|\sum_{\genfrac{}{}{0pt}{}{\xi_1,\xi_2 \in \Z}{\xi_1
            \not= \xi, \xi_2\not=\xi}} \int
        m_{k,j,\nu}(\mu,\mu_1,\mu_2)f_1(\mu_1)f_2(\mu_2)
        f_3(\mu_3) d\tau_1d\tau_2 \right\|_{{\ell}^{r'}_{\xi}L^2_{\tau}}\\
      \ls &
      \|f_1\|_{{\ell}^{q'}_{\xi}L^2_{\tau}}\|f_2\|_{{\ell}^{q'}_{\xi}L^2_{\tau}}
      \|f_3\|_{{\ell}^{r'}_{\xi}L^2_{\tau}}
    \end{split}
  \end{equation}
  for any $0 \leq \nu<\frac{1}{3q'}$, $k=0,\ldots,3$ and $j=1,2$,
  where
  \begin{align*}
    m_{k,1,\nu}(\mu,\mu_1,\mu_2)&= \frac{1}{
      \lb\xi_1\rb^{\frac{1}{2}}\lb\xi_2\rb^{\frac{1}{2}}
      \prod_{j=0,j\not=k}^3\lb\sigma_j\rb^{{\frac{1}{2}}-\nu}}\\
    m_{k,2,\nu}(\mu,\mu_1,\mu_2)& =\frac{1}{
      \lb\xi-\xi_1\rb^{\frac{1}{2}}\lb\xi-\xi_2\rb^{\frac{1}{2}}
      \prod_{j=0,j\not=k}^3\lb\sigma_j\rb^{{\frac{1}{2}}-\nu}}
  \end{align*}
  for $k=0,\ldots,3$.  Clearly, \eqref{eq:tri_mult_x_nu} implies
  \eqref{eq:tri_mult_x} with $\delta=3\nu$ because
  \begin{equation}\label{eq:gain_T}
    \left\|\lb \sigma_j \rb^{-\nu} f_j\right\|_{\ell^{p'}_\xi L^2_\tau} \ls
    T^{\nu}\|f_j\|_{\ell^{p'}_\xi L^2_\tau} \quad,  1 \leq p \leq \infty
  \end{equation}
  \noindent \emph{Case 0.1:} We consider the contribution
  \begin{align*}
    t_{0,1}:= & \left\|\sum_{\genfrac{}{}{0pt}{}{\xi_1,\xi_2 \in
          \Z}{\xi_1,\xi_2 \not= \xi}}\int \frac{f_1(\mu_1)}{\lb
        \xi_1\rb^{\frac{1}{2}}\lb\sigma_1\rb^{\frac{1}{2}-\nu}}
      \frac{f_2(\mu_2)}{\lb
        \xi_2\rb^{\frac{1}{2}}\lb\sigma_2\rb^{\frac{1}{2}-\nu}}
      \frac{f_3(\mu_3)}{\lb\sigma_3\rb^{\frac{1}{2}-\nu}}
      d\tau_1d\tau_2 \right\|_{{\ell}^{r'}_{\xi}L^2_{\tau}}\\
    \ls & \left\|\sum_{\genfrac{}{}{0pt}{}{\xi_1,\xi_2 \in
          \Z}{\xi_1,\xi_2 \not= \xi}}
      \lb\xi_1\rb^{-\frac{1}{2}}\lb\xi_2\rb^{-\frac{1}{2}}
      I(\mu,\mu_1,\mu_2) \left(\int f^2_1f^2_2f^2_3 d\tau_1
        d\tau_2\right)^{\frac{1}{2}}
    \right\|_{{\ell}^{r'}_{\xi}L^2_{\tau}}
  \end{align*}
  where
  \begin{equation*}
    I(\mu,\mu_1,\mu_2):=\left(\int \frac{d\tau_1 d\tau_2}
      {(\lb\sigma_1\rb\lb\sigma_2\rb\lb\sigma_3\rb)^{1-2\nu}}\right)^{\frac{1}{2}}\\
    \ls  \lb \sigma^{(0)}_{res}\rb^{\frac{1}{q'}-\frac{1}{2}-}
  \end{equation*}
  with $\sigma^{(0)}_{res}=\tau+\xi^2-2(\xi-\xi_1)(\xi-\xi_2)$ by two
  applications of Lemma \ref{lem:fam2.12}. H\"older's inequality in
  $\xi_1,\xi_2$ leads to
  \begin{equation*}
    t_{0,1}\ls \left\|\Sigma_{0,1}(\mu)\left(\sum_{\xi_1,\xi_2 \in \Z}
        \left(\int \frac{f^2_1(\mu_1)}{\lb\xi_1\rb^{1-}}
          \frac{f^2_2(\mu_2)}{\lb\xi_2\rb^{1-}}f^2_3(\mu_3) d\tau_1 d\tau_2
        \right)^{\frac{\varrho}{2}}
      \right)^{\frac{1}{\varrho}}
    \right\|_{{\ell}^{r'}_{\xi}L^2_{\tau}}
  \end{equation*}
  where
  \begin{equation*}
    \Sigma_{0,1}(\mu):=
    \left(\sum_{\genfrac{}{}{0pt}{}{\xi_1,\xi_2 \in \Z}{\xi_1,\xi_2 \not=
          \xi}}\lb\xi_1\rb^{0-}\lb\xi_2\rb^{0-}
      \lb \sigma^{(0)}_{res}\rb^{-1-}\right)^{\frac{1}{\varrho'}}
  \end{equation*}
  for $\varrho=\frac{2q'}{q'+2}+$ and $\varrho'=\frac{2q'}{q'-2}-$.
  The sum $\Sigma_{0,1}(\mu)$ is uniformly bounded due to Corollary
  \ref{cor:sums}, estimate \eqref{eq:sum1}. Hence,
  \begin{equation*}
    t_{0,1} \ls \left\|\left(\sum_{\xi_1,\xi_2 \in \Z}
        \frac{\|f_1(\cdot,\xi_1)\|^\varrho_{L^2}}{\lb\xi_1\rb^{\frac{\varrho}{2}-}}
        \frac{\|f_2(\cdot,\xi_2)\|^\varrho_{L^2}}{\lb\xi_2\rb^{\frac{\varrho}{2}-}}
        \|f_3(\cdot,\xi_3)\|^\varrho_{L^2}\right)^{\frac{1}{\varrho}}\right\|_{{\ell}^{r'}_{\xi}}
  \end{equation*}
  by Minkowski's inequality because $\varrho\leq 2$. Now, we apply
  H\"older's inequality to obtain
  \begin{equation}\label{eq:case12_st}
    t_{0,1} \ls \left\|\left(\sum_{\xi_1,\xi_2 \in \Z}
        \frac{\|f_1(\cdot,\xi_1)\|^{r'}_{L^2}}{\lb\xi_1\rb^{1-\frac{r'}{q'}+}}
        \frac{\|f_2(\cdot,\xi_2)\|^{r'}_{L^2}}{\lb\xi_2\rb^{1-\frac{r'}{q'}+}}
        \|f_3(\cdot,\xi_3)\|^{r'}_{L^2}\right)^{\frac{1}{r'}}
    \right\|_{{\ell}^{r'}_{\xi}}
  \end{equation}
  H\"older's inequality shows
  \begin{equation*}
    \left(\sum_{\xi_i \in \Z}\|f_i(\cdot,\xi_i)\|^{r'}_{L^2}\lb\xi_i\rb^{\frac{r'}{q'}-1-}
    \right)^{\frac{1}{r'}}
    \ls      \|f_i\|_{{\ell}^{q'}_{\xi}L^2_\tau} \;\;\quad , i=1,2
  \end{equation*}
  Hence, Fubini's theorem provides
  \begin{equation*}
    t_{0,1}\ls     \|f_1\|_{{\ell}^{q'}_{\xi}L^2_{\tau}}\|f_2\|_{{\ell}^{q'}_{\xi}L^2_{\tau}}
    \|f_3\|_{{\ell}^{r'}_{\xi}L^2_{\tau}}
  \end{equation*}
  for any $1<q\leq r\leq 2$, as desired.\\
  \noindent\emph{Case 0.2:}
  We consider the contribution
  \begin{equation*}
    t_{0,2}:=\left\|\sum_{\genfrac{}{}{0pt}{}{\xi_1,\xi_2 \in \Z}{\xi_1,\xi_2 \not= \xi}}\int
      \frac{f_1(\mu_1)}{\lb
        \xi-\xi_1\rb^{\frac{1}{2}}\lb\sigma_1\rb^{\frac{1}{2}-\nu}}
      \frac{f_2(\mu_2)}{\lb\xi-\xi_2\rb^{\frac{1}{2}}
        \lb\sigma_2\rb^{\frac{1}{2}-\nu}}
      \frac{f_3(\mu_3)}{\lb\sigma_3\rb^{\frac{1}{2}-\nu}}
      d\tau_1d\tau_2 \right\|_{{\ell}^{r'}_{\xi}L^2_{\tau}}
  \end{equation*}
  By replacing the weight $\lb \xi_1\rb^{-\frac{1}{2}}\lb
  \xi_2\rb^{-\frac{1}{2}}$ by $\lb \xi-\xi_1\rb^{-\frac{1}{2}}\lb
  \xi-\xi_2\rb^{-\frac{1}{2}}$ in the expression $t_{0,1}$, the same
  arguments as in the previous case lead to
  \begin{equation*}
    t_{0,2} \ls \left\|\left(\sum_{\xi_1,\xi_2 \in \Z}
        \frac{\|f_1(\cdot,\xi_1)\|^{r'}_{L^2}}{\lb\xi-\xi_1\rb^{1-\frac{r'}{q'}+}}
        \frac{\|f_2(\cdot,\xi_2)\|^{r'}_{L^2}}{\lb\xi-\xi_2\rb^{1-\frac{r'}{q'}+}}
        \|f_3(\cdot,\xi_3)\|^{r'}_{L^2}\right)^{\frac{1}{r'}}
    \right\|_{{\ell}^{r'}_{\xi}}
  \end{equation*}
  instead of \eqref{eq:case12_st} where we used Corollary
  \ref{cor:sums}, estimate \eqref{eq:sum01} to bound the sum
  \begin{equation*}
    \Sigma_{0,2}(\mu):=
    \left(\sum_{\genfrac{}{}{0pt}{}{\xi_1,\xi_2 \in \Z}{\xi_1,\xi_2 \not= \xi}}
      \lb\xi-\xi_1\rb^{0-}\lb\xi-\xi_2\rb^{0-}
      \lb \sigma^{(0)}_{res}\rb^{-1-}\right)^{\frac{1}{\varrho'}}
  \end{equation*}
  By the change of variables $\xi \mapsto \xi-\xi_1-\xi_2$ we obtain
  \begin{equation*}
    t_{0,2} \ls \left(\sum_{\xi,\xi_1,\xi_2 \in \Z}
      \frac{\|f_1(\cdot,\xi_1)\|^{r'}_{L^2}}{\lb\xi+\xi_1\rb^{1-\frac{r'}{q'}+}}
      \frac{\|f_2(\cdot,\xi_2)\|^{r'}_{L^2}}{\lb\xi+\xi_2\rb^{1-\frac{r'}{q'}+}}
      \|f_3(\cdot,\xi)\|^{r'}_{L^2}\right)^{\frac{1}{r'}}
  \end{equation*}
  Now, we sum first in $\xi_1, \xi_2$ and use
  \begin{equation}\label{eq:unif_emb}
    \sup_{\xi \in \Z}\left(\sum_{\xi_i \in \Z}\|f_i(\cdot,\xi_i)\|^{r'}_{L^2}
      \lb\xi+\xi_i\rb^{\frac{r'}{q'}-1-}
    \right)^{\frac{1}{r'}}
    \ls      \|f_i\|_{{\ell}^{q'}_{\xi}L^2_\tau} \;\;\quad , i=1,2
  \end{equation}
  to obtain
  \begin{equation*}
    t_{0,2}\ls     \|f_1\|_{{\ell}^{q'}_{\xi}L^2_{\tau}}\|f_2\|_{{\ell}^{q'}_{\xi}L^2_{\tau}}
    \|f_3\|_{{\ell}^{r'}_{\xi}L^2_{\tau}}
  \end{equation*}
  similarly as above.
  \\
  \noindent\emph{Case 1.1:} From now on we have to restrict ourselves
  to $2\leq q' <4$.  We use duality and consider for $\varphi \in
  \ell^{r}_\xi L^2_\tau$ the quantity $t_{1,1}$ defined by
  \begin{align*}
    & \sum_{\xi \in \Z} \int
    \frac{\varphi(\mu)}{\lb\sigma_0\rb^{\frac{1}{2}-\nu}}
    \sum_{\genfrac{}{}{0pt}{}{\xi_1,\xi_2 \in \Z}{\xi_1,\xi_2 \not=
        \xi}}\int \frac{f_1(\mu_1)}{\lb \xi_1\rb^{\frac{1}{2}}}
    \frac{f_2(\mu_2)}{\lb
      \xi_2\rb^{\frac{1}{2}}\lb\sigma_2\rb^{\frac{1}{2}-\nu}}
    \frac{f_3(\mu_3)}{\lb\sigma_3\rb^{\frac{1}{2}-\nu}}
    d\tau_1d\tau_2 d\tau\\
    = &\sum_{\xi_1 \in \Z} \int f_1(\mu_1)
    \sum_{\genfrac{}{}{0pt}{}{\xi,\xi_2 \in \Z}{\xi_1,\xi_2 \not=
        \xi}}\int
    \frac{\varphi(\mu)}{\lb\sigma_0\rb^{\frac{1}{2}-\nu}}
    \frac{f_2(\mu_2)}{\lb
      \xi_2\rb^{\frac{1}{2}}\lb\sigma_2\rb^{\frac{1}{2}-\nu}}
    \frac{f_3(\mu_3)}{\lb
      \xi_1\rb^{\frac{1}{2}}\lb\sigma_3\rb^{\frac{1}{2}-\nu}} d\tau
    d\tau_2 d\tau_1
  \end{align*}
  Cauchy-Schwarz in $\tau,\tau_2$ and two applications of Lemma
  \ref{lem:fam2.12} show
  \begin{equation*}
    t_{1,1}\ls \sum_{\xi_1 \in \Z} \int
    f_1(\mu_1)
    \sum_{\genfrac{}{}{0pt}{}{\xi,\xi_2 \in \Z}{\xi_1,\xi_2 \not= \xi}}
    \lb \sigma^{(1)}_{res} \rb^{\frac{1}{q'}-\frac{1}{2}-}
    \left(\int \frac{\varphi^{2}(\mu) f^2_2(\mu_2) f^2_3(\mu_3)}{\lb\xi_1\rb\lb\xi_2\rb}
      d\tau d\tau_2 \right)^{\frac{1}{2}}
    d\tau_1
  \end{equation*}
  where with
  $\sigma^{(1)}_{res}=\tau_1+\xi_1^2+2(\xi-\xi_1)(\xi-\xi_2)$.
  H\"older's inequality in $\xi,\xi_2$ leads to
  \begin{equation*}
    t_{1,1}\ls \sum_{\xi_1 \in \Z} \int
    f_1(\mu_1)\Sigma_{1,1}(\mu_1)
    \left(\sum_{\xi,\xi_2 \in \Z}
      \left( \int 
        \frac{\varphi^2(\mu)f^2_2(\mu_2)f^2_3(\mu_3)}{\lb \xi_1\rb^{1-}\lb \xi_2\rb^{1-}}
        d\tau d\tau_2
      \right)^{\frac{\varrho}{2}}\right)^{\frac{1}{\varrho}}d\tau_1
  \end{equation*}
  for $\varrho=\frac{2q'}{q'+2}+$ and $\varrho'=\frac{2q'}{q'-2}-$,
  where
  \begin{equation*}
    \Sigma_{1,1}(\mu_1):=
    \left(\lb\xi_1\rb^{0-} \sum_{\genfrac{}{}{0pt}{}{\xi,\xi_2 \in \Z}{\xi_1,\xi_2 \not= \xi}}
      \lb\xi_2\rb^{0-}
      \lb \sigma^{(1)}_{res}\rb^{-1-}\right)^{\frac{1}{\varrho'}}
  \end{equation*}
  This is bounded by Corollary \ref{cor:sums}, estimate
  \eqref{eq:sum2}. Cauchy-Schwarz in $\tau_1$ and Minkowski's
  inequality provide
  \begin{equation}\label{eq:case22_st}
    t_{1,1}\ls \sum_{\xi_1 \in \Z}\|f_1(\cdot, \xi_1)\|_{L^2_{\tau}}
    \left(\sum_{\xi,\xi_2}
      \frac{\|\varphi(\cdot,\xi)\|^\varrho_{L^2_\tau}\|f_2(\cdot,\xi_2)\|^\varrho_{L^2_\tau}
        \|f_3(\cdot,\xi_3)\|^\varrho_{L^2_\tau}}
      {\lb \xi_1\rb^{\frac{\varrho}{2}-}\lb\xi_2\rb^{\frac{\varrho}{2}-} }
    \right)^\frac{1}{\varrho}
  \end{equation}
  Now, we use H\"older in $\xi_1$ to obtain
  \begin{equation}
    t_{1,1}\ls \|f_1\|_{\ell^{q'}_{\xi} L^2_{\tau}}
    \left\|\left(\sum_{\xi,\xi_2}
        \frac{\|f_2(\cdot,\xi_2)\|^\varrho_{L^2_\tau}\|f_3(\cdot,\xi_3)\|^\varrho_{L^2_\tau}
          \|\varphi(\cdot,\xi)\|^{\varrho}_{L^2_\tau}}
        {\lb \xi_2\rb^{\frac{\varrho}{2}-}}
      \right)^{\frac{1}{\varrho}}
    \right\|_{\ell^{\varrho'}_{\xi_1}}
  \end{equation}
  By the change of variables $\xi_2\mapsto -\xi_2, \xi \mapsto \xi$ the
  second factor equals
  \begin{equation*}
    \left\|\sum_{\xi,\xi_2}
      \frac{\|\tilde{f}_2(\cdot,\xi_2)\|^\varrho_{L^2_\tau}}{\lb \xi_2\rb^{\frac{\varrho}{2}-}}
      \|\varphi(\cdot,\xi)\|^{\varrho}_{L^2_\tau}
      \|\tilde{f}_3(\cdot,\xi_1-\xi_2-\xi)\|^\varrho_{L^2_\tau}
    \right\|_{\ell^{\frac{\varrho'}{\varrho}}_{\xi_1}}^{\frac{1}{\varrho}}
  \end{equation*}
  where $\tilde{f}_j=f_j(-\cdot,-\cdot)$, $j=2,3$. This convolution is
  bounded by
  \begin{equation*}
    \|\frac{\tilde{f}_2}{\lb \xi_2\rb^{\frac{1}{2}-}}\|_{\ell^{\varrho }_{\xi_2} L^2_\tau}
    \|\varphi\|_{\ell^{r}_\xi L^2_\tau}
    \|\tilde{f}_3\|_{\ell^{r'}_\xi L^2_\tau}
  \end{equation*}
  due to Young's inequality, because
  $$
  2+\frac{\varrho}{\varrho'}= 1+\frac{\varrho}{r}+\frac{\varrho}{r'}
  $$
  Another application of H\"older's inequality with respect to $f_2$
  yields
  \begin{equation*}
    t_{1,1}\ls
    \|\varphi\|_{\ell^{r}_\xi L^2_\tau}
    \|f_1\|_{\ell^{q'}_{\xi} L^2_\tau}\|f_2\|_{\ell^{q'}_{\xi} L^2_\tau}
    \|f_3\|_{\ell^{r'}_\xi L^2_\tau}
  \end{equation*}
  for $4/3<q\leq r\leq 2$.\\
  \noindent\emph{Case 1.2:}
  For the contribution $t_{1,2}$ the same approach as above leads to
  \begin{equation*}
    t_{1,2}\ls \sum_{\xi_1 \in \Z}\|f_1(\cdot, \xi_1)\|_{L^2_{\tau}}
    \left(\sum_{\xi,\xi_2}
      \frac{\|\varphi(\cdot,\xi)\|^\varrho_{L^2_\tau}
        \|f_2(\cdot,\xi_2)\|^\varrho_{L^2_\tau}\|f_3(\cdot,\xi_3)\|^\varrho_{L^2_\tau}}
      {\lb \xi- \xi_1\rb^{\frac{\varrho}{2}-}\lb\xi-\xi_2\rb^{\frac{\varrho}{2}-} }
    \right)^\frac{1}{\varrho}
  \end{equation*}
  instead of \eqref{eq:case22_st} by replacing $\lb\xi_1\rb$,
  $\lb\xi_2\rb$ in $t_{1,1}$ by $\lb\xi-\xi_1\rb$, $\lb\xi-\xi_2\rb$,
  respectively. The only difference is the use of Corollary
  \ref{cor:sums}, estimate \eqref{eq:sum01} to bound the sum
  \begin{equation*}
    \Sigma_{1,2}(\mu_1):=
    \left(\sum_{\genfrac{}{}{0pt}{}{\xi,\xi_2 \in \Z}{\xi_1,\xi_2 \not= \xi}}
      \lb\xi-\xi_1\rb^{0-}\lb\xi-\xi_2\rb^{0-}
      \lb \sigma^{(1)}_{res}\rb^{-1-}\right)^{\frac{1}{\varrho'}}
  \end{equation*}
  H\"older's inequality in $\xi_1$ and then in $\xi,\xi_2$ provides
  \begin{align*}
    t_{1,2} \ls& \|f_1\|_{\ell^{q'}_{\xi} L^2_{\tau}}
    \left\|\left(\sum_{\xi,\xi_2}
        \frac{\|f_2(\cdot,\xi_2)\|^\varrho_{L^2_\tau}\|f_3(\cdot,\xi_3)\|^\varrho_{L^2_\tau}
          \|\varphi(\cdot,\xi)\|^{\varrho}_{L^2_\tau}} {\lb\xi-
          \xi_1\rb^{\frac{\varrho}{2}-}\lb\xi-\xi_2\rb^{\frac{\varrho}{2}-}}
      \right)^{\frac{1}{\varrho}}
    \right\|_{\ell^{q}_{\xi_1}}\\
    \ls &\|f_1\|_{\ell^{q'}_{\xi} L^2_{\tau}}
    \left(\sum_{\xi,\xi_1,\xi_2 \in \Z}
      \frac{\|\varphi(\cdot,\xi)\|^{q}_{L^2}}{\lb\xi-\xi_1\rb^{1-\frac{q}{q'}+}}
      \frac{\|f_2(\cdot,\xi_2)\|^{q}_{L^2}}{\lb\xi-\xi_2\rb^{1-\frac{q}{q'}+}}
      \|f_3(\cdot,\xi_3)\|^{q}_{L^2}\right)^{\frac{1}{q}}
  \end{align*}
  Now, H\"older in $\xi$ and the change of variables $\xi_1\mapsto
  \xi-\xi_1$ gives
  \begin{align*}
    &\left(\sum_{\xi,\xi_1,\xi_2 \in \Z}
      \frac{\|\varphi(\cdot,\xi)\|^{q}_{L^2}}{\lb\xi-\xi_1\rb^{1-\frac{q}{q'}+}}
      \frac{\|f_2(\cdot,\xi_2)\|^{q}_{L^2}}{\lb\xi-\xi_2\rb^{1-\frac{q}{q'}+}}
      \|f_3(\cdot,\xi_3)\|^{q}_{L^2}\right)^{\frac{1}{q}}\\
    \ls&\|\varphi\|_{\ell^{r}_\xi L^2_\tau}\left\|\sum_{\xi_2 \in \Z}
      \frac{\|f_2(\cdot,\xi_2)\|^{q}_{L^2}}{\lb\xi-\xi_2\rb^{1-\frac{q}{q'}+}}
      \sum_{\xi_1 \in \Z}
      \frac{\|f_3(\cdot,\xi_1-\xi_2)\|^{q}_{L^2}}{\lb\xi_1\rb^{1-\frac{q}{q'}+}}
    \right\|_{\ell^{\frac{r}{r-q}}_\xi}^{\frac{1}{q}}
  \end{align*}
  Let us define
  \begin{equation*}
    \psi(\xi_2)=\sum_{\xi_1 \in \Z}
    \frac{\|f_3(\cdot,\xi_1-\xi_2)\|^{q}_{L^2}}{\lb\xi_1\rb^{1-\frac{q}{q'}+}}\\
  \end{equation*}
  Young's inequality shows
  \begin{equation*}
    \|\psi\|_{\ell^{\frac{r}{r-q}}}\ls \|f_3\|^{q}_{\ell^{r'}_\xi L^2_\tau}
  \end{equation*}
  and therefore
  \begin{equation*}
    \left\|\sum_{\xi_2 \in \Z}
      \frac{\|f_2(\cdot,\xi_2)\|^{q}_{L^2}}{\lb\xi-\xi_2\rb^{1-\frac{q}{q'}+}}
      \psi(\xi_2)
    \right\|_{\ell^{\frac{r}{r-q}}_\xi}^{\frac{1}{q}}
    \ls  \|f_3\|_{\ell^{r'}_\xi L^2_\tau}\left\|\sum_{\xi_2 \in \Z}
      \frac{\|f_2(\cdot,\xi_2)\|^{r}_{L^2}}{\lb\xi-\xi_2\rb^{\frac{r}{q}-\frac{r}{q'}+}}
    \right\|_{\ell^{\frac{q}{r-q}}_\xi}^{\frac{1}{r}}
  \end{equation*}
  by H\"older's inequality in $\xi_2$. Due to the fact that
  \begin{equation*}
    1+\frac{r-q}{q} = \frac{r}{q'}+\frac{r}{q}-\frac{r}{q'}
  \end{equation*}
  Young's inequality shows that
  \begin{equation*}
    t_{1,2}\ls     \|\varphi\|_{\ell^{r}_\xi L^2_\tau}
    \|f_1\|_{{\ell}^{q'}_{\xi}L^2_{\tau}}\|f_2\|_{{\ell}^{q'}_{\xi}L^2_{\tau}}
    \|f_3\|_{{\ell}^{r'}_{\xi}L^2_{\tau}}
  \end{equation*}
  for all $4/3<q\leq r\leq 2$.\\
  \noindent\emph{Case 2.1:} To control the contribution from
  $m_{2,1,\nu}$ we exchange the roles of $f_1$ and $f_2$ and the
  arguments from Case 1.1
  apply.\\
  \noindent\emph{Case 2.2:} To control the contribution from
  $m_{2,2,\nu}$ we exchange the roles of $f_1$ and $f_2$ and the
  arguments from Case 1.2
  apply.\\
  \noindent\emph{Case 3.1:} Fix $ 2 \leq q' < 4$ and
  $0 \leq \nu<\frac{1}{3q'}$. By the change of variables
  \begin{equation}\label{eq:cov}
    \mu_1 \mapsto -\mu_1, \mu_2 \mapsto -\mu_2, \mu \mapsto \mu-\mu_1-\mu_2
  \end{equation}
  we obtain for the contribution $t_{3,1}$ the identity
  \begin{align*}
    &\sum_{\xi \in \Z} \int
    \frac{\varphi(\mu)}{\lb\sigma_0\rb^{\frac{1}{2}-\nu}}
    \sum_{\genfrac{}{}{0pt}{}{\xi_1,\xi_2 \in \Z}{\xi_1,\xi_2 \not=
        \xi}} \int\int
    \frac{f_1(\mu_1)}{\lb\xi_1\rb^{\frac{1}{2}}\lb\sigma_1\rb^{\frac{1}{2}-\nu}}
    \frac{f_2(\mu_2)}{\lb\xi_2\rb^{\frac{1}{2}}\lb\sigma_2\rb^{\frac{1}{2}-\nu}}
    f_3(\mu_3)
    d\tau_1d\tau_2 d\tau\\
    =&\sum_{\xi \in \Z} \int f_3(\mu)
    \sum_{\genfrac{}{}{0pt}{}{\xi_1,\xi_2 \in \Z}{\xi_1,\xi_2 \not=
        \xi}} \int\int
    \frac{\tilde{f}_1(\mu_1)}{\lb\xi_1\rb^{\frac{1}{2}}
      \lb\tilde{\sigma}_1\rb^{\frac{1}{2}-\nu}}
    \frac{\tilde{f}_2(\mu_2)}{\lb\xi_2\rb^{\frac{1}{2}}
      \lb\tilde{\sigma}_2\rb^{\frac{1}{2}-\nu}}
    \frac{\varphi(\mu_3)}{\lb\tilde{\sigma}_3\rb^{\frac{1}{2}-\nu}}
    d\tau_1d\tau_2 d\tau
  \end{align*}
  where $\tilde{\sigma}_1=\tau_1-\xi_1^2$,
  $\tilde{\sigma}_2=\tau_2-\xi_2^2$,
  $\tilde{\sigma}_3=\tau-\tau_1-\tau_2+(\xi-\xi_1-\xi_2)^2$ and
  $\tilde{f}_j=f_j(-\cdot,-\cdot)$, $j=1,2$. Using Cauchy-Schwarz in
  $\tau_1,\tau_2$ and Lemma \ref{lem:fam2.12}, the quantity $t_{3,1}$
  is bounded by
  \begin{equation*}
    \sum_{\xi \in \Z} \int f_3(\mu) \sum_{\genfrac{}{}{0pt}{}{\xi_1,\xi_2 \in
        \Z}{\xi_1,\xi_2 \not= \xi}}
    \lb\sigma^{(3)}_{res}
    \rb^{\frac{1}{q'}-\frac{1}{2}-}
    \left(\int\int
      \frac{\tilde{f}_1^2(\mu_1)
        \tilde{f}_2^2(\mu_2)
        \varphi^2(\mu_3)}{\lb\xi_1\rb\lb\xi_2\rb}
      d\tau_1d\tau_2 \right)^{\frac{1}{2}}
    d\tau
  \end{equation*}
  where $\sigma^{(3)}_{res}=\tau-\xi^2+2(\xi-\xi_1)(\xi-\xi_2)$.
  H\"older's inequality leads to
  \begin{equation*}
    \sum_{\xi \in \Z} \int f_3(\mu) \Sigma_{3,1}(\mu)
    \left(\sum_{\genfrac{}{}{0pt}{}{\xi_1,\xi_2 \in
          \Z}{\xi_1,\xi_2 \not= \xi}}  \left(\int\int
        \frac{\tilde{f}_1^2(\mu_1)
          \tilde{f}_2^2(\mu_2)
          \varphi^2(\mu_3)}{\lb\xi_1\rb^{1-}\lb\xi_2\rb^{1-}}
        d\tau_1d\tau_2 \right)^{\frac{\varrho}{2}}\right)^{\frac{1}{\varrho}}
    d\tau
  \end{equation*}
  as an upper bound for $t_{3,1}$ with
  \begin{equation*}
    \Sigma_{3,1}(\mu)=\left(\sum_{\genfrac{}{}{0pt}{}{\xi_1,\xi_2 \in
          \Z}{\xi_1,\xi_2 \not= \xi}} \lb\xi_1\rb^{0-}
      \lb\xi_2\rb^{0-}\lb\sigma^{(3)}_{res}\rb^{-1-}\right)^{\frac{1}{\varrho'}}
  \end{equation*}
  which is uniformly bounded by Corollary \ref{cor:sums}, estimate
  \eqref{eq:sum1}. By Cauchy-Schwarz in $\tau$ and Minkowski's
  inequality $t_{3,1}$ is dominated by
  \begin{equation*}
    \sum_{\xi \in \Z} \|f_3(\cdot,\xi)\|_{L^2_\tau}
    \left(\sum_{\genfrac{}{}{0pt}{}{\xi_1,\xi_2 \in
          \Z}{\xi_1,\xi_2 \not= \xi}}
      \frac{\|\tilde{f}_1(\cdot,\xi_1)\|^\varrho_{L^2_\tau}
        \|\tilde{f}_2(\cdot,\xi_2)\|^\varrho_{L^2_\tau}
        \|\varphi(\cdot,\xi_3)\|^\varrho_{L^2_\tau}}
      {\lb\xi_1\rb^{\frac{\varrho}{2}-}\lb\xi_2\rb^{\frac{\varrho}{2}-}}
    \right)^{\frac{1}{\varrho}}
  \end{equation*}
  Now, we recall that $r \geq \varrho$ for all $4/3<q\leq r\leq 2$ and
  apply H\"older's inequality and Fubini to obtain
  \begin{align}
    t_{3,1} \ls& \|f_3\|_{\ell^{r'}_\xi L^2_\tau}
    \left\|\left(\sum_{\genfrac{}{}{0pt}{}{\xi_1,\xi_2 \in
            \Z}{\xi_1,\xi_2 \not= \xi}}
        \frac{\|\tilde{f}_1(\cdot,\xi_1)\|^\varrho_{L^2_\tau}
          \|\tilde{f}_2(\cdot,\xi_2)\|^\varrho_{L^2_\tau}
          \|\varphi(\cdot,\xi_3)\|^\varrho_{L^2_\tau}}
        {\lb\xi_1\rb^{\frac{\varrho}{2}-}\lb\xi_2\rb^{\frac{\varrho}{2}-}}
      \right)^{\frac{1}{\varrho}}
    \right\|_{{\ell}^{r}_{\xi}}\nonumber \\
    \ls & \|f_3\|_{\ell^{r'}_\xi L^2_\tau}
    \left(\sum_{\genfrac{}{}{0pt}{}{\xi,\xi_1,\xi_2 \in
          \Z}{\xi_1,\xi_2 \not= \xi}}
      \frac{\|\tilde{f}_1(\cdot,\xi_1)\|^{r}_{L^2_\tau}
        \|\tilde{f}_2(\cdot,\xi_2)\|^{r}_{L^2_\tau}
        \|\varphi(\cdot,\xi_3)\|^{r}_{L^2_\tau}}
      {\lb\xi_1\rb^{1-\frac{r}{q'}+}\lb\xi_2\rb^{1-\frac{r}{q'}+}}
    \right)^{\frac{1}{r}} \label{eq:case32_st}
  \end{align}
  Again, H\"older's inequality shows
  \begin{equation*}
    \left(\sum_{\xi_i \in \Z}\|\tilde{f}_i(\cdot,\xi_i)\|^{r}_{L^2}\lb\xi_i\rb^{\frac{r}{q'}-1-}
    \right)^{\frac{1}{r}}
    \ls      \|f_i\|_{{\ell}^{q'}_{\xi}L^2_\tau} \;\;\quad , i=1,2
  \end{equation*}
  Hence,
  \begin{equation*}
    t_{3,1}\ls     \|f_1\|_{{\ell}^{q'}_{\xi}L^2_{\tau}}\|f_2\|_{{\ell}^{q'}_{\xi}L^2_{\tau}}
    \|f_3\|_{{\ell}^{r'}_{\xi}L^2_{\tau}}\|\varphi\|_{{\ell}^{r}_{\xi}L^2_{\tau}}
  \end{equation*}
  for $4/3<q\leq r\leq 2$, as desired.\\
  \noindent\emph{Case 3.2:} To obtain the contribution $t_{3,2}$ we
  replace $\lb\xi_1\rb$ and $\lb\xi_2\rb$ in $t_{3,1}$ by
  $\lb\xi-\xi_1\rb$ and $\lb\xi-\xi_2\rb$, respectively.  The change
  of variables \eqref{eq:cov} transforms $\lb\xi-\xi_1\rb$ to
  $\lb\xi-\xi_2\rb$ and vice versa and we follow the arguments above
  to obtain
  \begin{equation*}
    t_{3,2}\ls \|f_3\|_{\ell^{r'}_\xi L^2_\tau} \left(\sum_{\genfrac{}{}{0pt}{}{\xi,\xi_1,\xi_2 \in
          \Z}{\xi_1,\xi_2 \not= \xi}}
      \frac{\|\tilde{f}_1(\cdot,\xi_1)\|^{r}_{L^2_\tau}
        \|\tilde{f}_2(\cdot,\xi_2)\|^{r}_{L^2_\tau}
        \|\varphi(\cdot,\xi_3)\|^{r}_{L^2_\tau}}
      {\lb\xi-\xi_1\rb^{1-\frac{r}{q'}+}\lb\xi-\xi_2\rb^{1-\frac{r}{q'}+}}
    \right)^{\frac{1}{r}}
  \end{equation*}
  instead of \eqref{eq:case32_st}, with the only exception that
  \begin{equation*}
    \Sigma_{3,2}(\mu):=
    \left(\sum_{\genfrac{}{}{0pt}{}{\xi_1,\xi_2 \in \Z}{\xi_1,\xi_2 \not= \xi}}
      \lb\xi-\xi_1\rb^{0-}\lb\xi-\xi_2\rb^{0-}
      \lb \sigma^{(3)}_{res}\rb^{-1-}\right)^{\frac{1}{\varrho'}}
  \end{equation*}
  is controlled by Corollary \ref{cor:sums}, estimate
  \eqref{eq:sum01}. Similar to the Case 0.2, the change of variables
  $\xi \mapsto \xi-\xi_1-\xi_2$ and H\"older's inequality
  \begin{equation*}
    \sup_{\xi\in \Z}
    \left(\sum_{\xi_i \in \Z}\|f_i(\cdot,\xi_i)\|^{r}_{L^2}\lb\xi+\xi_i\rb^{\frac{r}{q'}-1-}
    \right)^{\frac{1}{r}}
    \ls      \|f_i\|_{{\ell}^{q'}_{\xi}L^2_\tau} \;\;\quad , i=1,2
  \end{equation*}
  yield
  \begin{equation*}
    t_{3,2}\ls \|f_1\|_{{\ell}^{q'}_{\xi}L^2_{\tau}}\|f_2\|_{{\ell}^{q'}_{\xi}L^2_{\tau}}
    \|f_3\|_{{\ell}^{r'}_{\xi}L^2_{\tau}}\|\varphi\|_{{\ell}^{r}_{\xi}L^2_{\tau}}
  \end{equation*}
  and the estimate for $T^\ast$ is done.  Finally, we consider the
  harmless contribution from $T^{\ast\ast}$ and show the much stronger
  estimate
  \begin{equation}\label{eq:add_est}
    \|T^{\ast\ast}(u_1,u_2,u_3)\|_{X^{\frac{1}{2},0}_{r,2}}\ls
    \|u_1\|_{X^{\frac{1}{2},\frac{1}{3}+}_{1,2}}\|u_2\|_{X^{\frac{1}{2},\frac{1}{3}+}_{1,2}}
    \|u_3\|_{X^{\frac{1}{2},\frac{1}{3}+}_{r,2}}
  \end{equation}
  which immediately yields the desired estimate by trivial embeddings
  and \eqref{eq:gain_T}.  Indeed, Young's and H\"older's inequality
  provide
  \begin{align*}
    & \left\|\lb\xi \rb^{\frac{1}{2}} \xi \int
      \frac{f_1(\tau_1,\xi)}{\lb\xi \rb^{\frac{1}{2}}\lb
        \tau_1+\xi^2\rb^{\frac{1}{3}+}} \frac{f_2(\tau_2,\xi)}{\lb\xi
        \rb^{\frac{1}{2}}\lb \tau_2+\xi^2\rb^{\frac{1}{3}+}} \frac{
        f_3(\tau_3,-\xi)}{\lb\xi \rb^{\frac{1}{2}}\lb
        \tau_3-\xi^2\rb^{\frac{1}{3}+}} d\tau_1d\tau_2
    \right\|_{{\ell}^{r'}_{\xi}L^2_{\tau}}\\
    \ls & \left\| \left\|\frac{f_1(\tau_1,\xi)}{\lb
          \tau_1+\xi^2\rb^{\frac{1}{3}+}}
      \right\|_{L^{\frac{6}{5}}_{\tau_1}}
      \left\|\frac{f_2(\tau_2,\xi)}{\lb
          \tau_2+\xi^2\rb^{\frac{1}{3}+}}\right\|_{L^{\frac{6}{5}}_{\tau_2}}
      \left\|\frac{f_3(\tau_3,-\xi)}{\lb \tau_3-\xi^2\rb^{\frac{1}{3}+}}
      \right\|_{L^{\frac{6}{5}}_{\tau_3}}
    \right\|_{{\ell}^{r'}_{\xi}}\\
    \ls &\|f_1\|_{\ell^{\infty}_\xi
      L^{2}_{\tau}}\|f_2\|_{\ell^{\infty}_\xi L^{2}_{\tau}}
    \|f_3\|_{\ell^{r'}_\xi L^{2}_{\tau}}
  \end{align*}
  This concludes the proof of Theorem \ref{thm:tri_x1}
\end{proof}
\begin{proof}[Proof of Theorem \ref{thm:tri_x2}]
  Due to the emdedding \eqref{eq:x_emb} the estimate for
  $T^{\ast\ast}$ is already covered by \eqref{eq:add_est}.  With the
  same notation as above the estimate \eqref{eq:tri_x2} for the
  contribution $T^\ast$ is equivalent to
  \begin{equation}\label{eq:tri_mult_y}
    \begin{split}
      &\left\|\sum_{\genfrac{}{}{0pt}{}{\xi_1,\xi_2 \in \Z}{\xi_1
            \not= \xi, \xi_2\not=\xi}}\int
        n(\mu,\mu_1,\mu_2)f_1(\mu_1)f_2(\mu_2)
        f_3(\mu_3) d\tau_1d\tau_2 \right\|_{{\ell}^{r'}_{\xi}L^1_{\tau}}\\
      \ls& T^\delta \|f_1\|_{{\ell}^{q'}_{\xi}L^2_{\tau}}
      \|f_2\|_{{\ell}^{q'}_{\xi}L^2_{\tau}}\|f_3\|_{{\ell}^{r'}_{\xi}L^2_{\tau}}
    \end{split}
  \end{equation}
  where $n=\lb\sigma_0\rb^{-\frac{1}{2}}m$. We decompose
  $n_{k,j}=\lb\sigma_0\rb^{-\frac{1}{2}}m_{k,j}$ as above. Again, due
  to the embedding \eqref{eq:x_emb} the stronger estimate
  \eqref{eq:tri_mult_x_nu} already proves estimate
  \eqref{eq:tri_mult_y} for $n$ replaced by $n_{k,j,\nu}$ with
  $k=1,2,3$, $j=1,2$, corresponding to the Cases 1-3 above. Hence, it
  is enough to consider the case $k=0$ where $\lb \sigma_0 \rb$ is the
  maximal modulation.
  \\
  \noindent\emph{Case 0.1:}
  Let us fix $1<q\leq r \leq 2$ and $0 \leq \delta<\frac{1}{q'}$. We
  proceed similarly to the Case 0.1 in the proof of Theorem
  \ref{thm:tri_x1}:
  \begin{align*}
    \widetilde{t_{0,1}}:= & \left\|\lb\sigma_{0}\rb^{-\frac{1}{2}}
      \sum_{\genfrac{}{}{0pt}{}{\xi_1,\xi_2 \in \Z}{\xi_1,\xi_2 \not=
          \xi}}\int \frac{f_1(\mu_1)}{\lb
        \xi_1\rb^{\frac{1}{2}}\lb\sigma_1\rb^{\frac{1}{2}-\nu}}
      \frac{f_2(\mu_2)}{\lb
        \xi_2\rb^{\frac{1}{2}}\lb\sigma_2\rb^{\frac{1}{2}-\nu}}
      \frac{f_3(\mu_3)}{\lb\sigma_3\rb^{\frac{1}{2}-\nu}}
      d\tau_1d\tau_2 \right\|_{{\ell}^{r'}_{\xi}L^1_{\tau}}\\
    \ls& \left\| \sum_{\genfrac{}{}{0pt}{}{\xi_1,\xi_2 \in
          \Z}{\xi_1,\xi_2 \not= \xi}}\int \frac{g_1(\mu_1)}{\lb
        \xi_1\rb^{\frac{1}{2}}\lb\sigma_1\rb^{\frac{1}{2}-\nu-}}
      \frac{g_2(\mu_2)}{\lb
        \xi_2\rb^{\frac{1}{2}}\lb\sigma_2\rb^{\frac{1}{2}-\nu-}}
      \frac{g_3(\mu_3)}{\lb\sigma_3\rb^{\frac{1}{2}-\nu-}}
      d\tau_1d\tau_2 \right\|_{{\ell}^{r'}_{\xi}L^p_{\tau}}
  \end{align*}
  for any $p=2-$, where we define $g_j=\lb\sigma_j\rb^{0-}f_j$ such
  that
  \begin{equation}\label{eq:prop_g}
    \|g_j\|_{{\ell}^{q'}_{\xi}L^p_{\tau}} \ls \|f_j\|_{{\ell}^{q'}_{\xi}L^2_{\tau}}
  \end{equation}
  Now, by H\"older's inequality and two applications of Lemma
  \ref{lem:fam2.12} we get
  \begin{align*}
    \widetilde{t_{0,1}} \ls &
    \left\|\sum_{\genfrac{}{}{0pt}{}{\xi_1,\xi_2 \in \Z}{\xi_1,\xi_2
          \not= \xi}}
      \lb\xi_1\rb^{-\frac{1}{2}}\lb\xi_2\rb^{-\frac{1}{2}} \lb
      \sigma^{(0)}_{res}\rb^{\frac{1}{q'}-\frac{1}{2}-} \left(\int
        g^p_1g^p_2g^p_3 d\tau_1 d\tau_2\right)^{\frac{1}{p}}
    \right\|_{{\ell}^{r'}_{\xi}L^p_{\tau}}
  \end{align*}
  with $\sigma^{(0)}_{res}=\tau+\xi^2-2(\xi-\xi_1)(\xi-\xi_2)$.
  H\"older's inequality in $\xi_1,\xi_2$ leads to
  \begin{equation*}
    \widetilde{t_{0,1}}\ls \left\|\widetilde{\Sigma_{0,1}}(\mu)
      \left(\sum_{\xi_1,\xi_2 \in \Z}
        \left(\int \frac{g^p_1(\mu_1)}{\lb\xi_1\rb^{1-}}
          \frac{g^p_2(\mu_2)}{\lb\xi_2\rb^{1-}}g^p_3(\mu_3) d\tau_1 d\tau_2
        \right)^{\frac{\varrho}{2}}
      \right)^{\frac{1}{\varrho}}
    \right\|_{{\ell}^{r'}_{\xi}L^p_{\tau}}
  \end{equation*}
  where
  \begin{equation*}
    \widetilde{\Sigma_{0,1}}(\mu):=
    \left(\sum_{\genfrac{}{}{0pt}{}{\xi_1,\xi_2 \in \Z}{\xi_1,\xi_2 \not=
          \xi}}\lb\xi_1\rb^{0-}\lb\xi_2\rb^{0-}
      \lb \sigma^{(0)}_{res}\rb^{-1-}\right)^{\frac{1}{\varrho'}}
  \end{equation*}
  for $\varrho=\frac{2q'}{q'+2}+$ and $\varrho'=\frac{2q'}{q'-2}-$.
  The sum $\widetilde{\Sigma_{0,1}}(\mu)$ is uniformly bounded due to
  Corollary \ref{cor:sums}, estimate \eqref{eq:sum1}. Hence,
  \begin{equation*}
    \widetilde{t_{0,1}} \ls \left\|\left(\sum_{\xi_1,\xi_2 \in \Z}
        \frac{\|g_1(\cdot,\xi_1)\|^\varrho_{L^p}}{\lb\xi_1\rb^{\frac{\varrho}{2}-}}
        \frac{\|g_2(\cdot,\xi_2)\|^\varrho_{L^p}}{\lb\xi_2\rb^{\frac{\varrho}{2}-}}
        \|g_3(\cdot,\xi_3)\|^\varrho_{L^p}\right)^{\frac{1}{\varrho}}
    \right\|_{{\ell}^{r'}_{\xi}}
  \end{equation*}
  by Minkowski's inequality because $\varrho\leq p$. Now, we apply
  H\"older's inequality and Fubini as in Case 0.1 of the proof of
  Theorem \ref{thm:tri_x1} and obtain
  \begin{equation*}
    \widetilde{t_{0,1}}\ls \|g_1\|_{{\ell}^{q'}_{\xi}L^p_{\tau}}
    \|g_2\|_{{\ell}^{q'}_{\xi}L^p_{\tau}}
    \|g_3\|_{{\ell}^{r'}_{\xi}L^p_{\tau}}
  \end{equation*}
  for any $1<q \leq r\leq 2$. Finally, \eqref{eq:prop_g} proves
  the desired estimate.\\
  \noindent\emph{Case 0.2:}
  We consider $\widetilde{t_{0,2}}$ defined as
  \begin{equation*}
    \left\|\lb\sigma_0\rb^{-\frac{1}{2}}
      \sum_{\genfrac{}{}{0pt}{}{\xi_1,\xi_2 \in \Z}{\xi_1,\xi_2 \not= \xi}}\int
      \frac{f_1(\mu_1)}{\lb
        \xi-\xi_1\rb^{\frac{1}{2}}\lb\sigma_1\rb^{\frac{1}{2}-\nu}}
      \frac{f_2(\mu_2)}{\lb\xi-\xi_2\rb^{\frac{1}{2}}
        \lb\sigma_2\rb^{\frac{1}{2}-\nu}}
      \frac{f_3(\mu_3)}{\lb\sigma_3\rb^{\frac{1}{2}-\nu}}
      d\tau_1d\tau_2 \right\|_{{\ell}^{q'}_{\xi}L^1_{\tau}}
  \end{equation*}
  The same arguments as in the previous case lead to
  \begin{equation*}
    \widetilde{t_{0,2}} \ls \left\|\left(\sum_{\xi_1,\xi_2 \in \Z}
        \frac{\|g_1(\cdot,\xi_1)\|^\varrho_{L^p}}
        {\lb\xi-\xi_1\rb^{\frac{\varrho}{2}-}}
        \frac{\|g_2(\cdot,\xi_2)\|^\varrho_{L^p}}
        {\lb\xi-\xi_2\rb^{\frac{\varrho}{2}-}}
        \|g_3(\cdot,\xi_3)\|^\varrho_{L^p}\right)^{\frac{1}{\varrho}}
    \right\|_{{\ell}^{r'}_{\xi}}
  \end{equation*}
  where we used Corollary \ref{cor:sums}, estimate \eqref{eq:sum01} to
  bound the sum
  \begin{equation*}
    \widetilde{\Sigma_{0,2}}(\mu):=
    \left(\sum_{\genfrac{}{}{0pt}{}{\xi_1,\xi_2 \in \Z}{\xi_1,\xi_2 \not= \xi}}
      \lb\xi-\xi_1\rb^{0-}\lb\xi-\xi_2\rb^{0-}
      \lb \sigma^{(0)}_{res}\rb^{-1-}\right)^{\frac{1}{\varrho'}}
  \end{equation*}
  By the change of variables $\xi \mapsto \xi-\xi_1-\xi_2$ we obtain
  \begin{equation*}
    \widetilde{t_{0,2}} \ls \left(\sum_{\xi,\xi_1,\xi_2 \in \Z}
      \frac{\|g_1(\cdot,\xi_1)\|^{r'}_{L^p}}{\lb\xi+\xi_1\rb^{1-\frac{r'}{q'}+}}
      \frac{\|g_2(\cdot,\xi_2)\|^{r'}_{L^p}}{\lb\xi+\xi_2\rb^{1-\frac{r'}{q'}+}}
      \|g_3(\cdot,\xi)\|^{r'}_{L^p}\right)^{\frac{1}{r'}}
  \end{equation*}
  Now, we sum first in $\xi_1, \xi_2$ and use the analogue of
  \eqref{eq:unif_emb} for $\|g_i\|_{L^p_\tau}$ ($i=1,2$) to obtain
  \begin{equation*}
    \widetilde{t_{0,2}}\ls  \|g_1\|_{{\ell}^{q'}_{\xi}L^p_{\tau}}
    \|g_2\|_{{\ell}^{q'}_{\xi}L^p_{\tau}}
    \|g_3\|_{{\ell}^{r'}_{\xi}L^p_{\tau}}
  \end{equation*}
  and recall the property \eqref{eq:prop_g} of $g_j$.
\end{proof}

\begin{proof}[Proof of Remark \ref{rem:counter}]
  Assume that the estimate \eqref{eq:counter} is valid for some $b\leq
  0$, $1 \leq r \leq \frac{4}{3}$ and $1\leq p,q \leq \infty$.  Then
  for all $f_i \in \ell^{r'}_{\xi}L^{q'}_{\tau}$ ($1 \leq i \leq 3$)
  and $f_0 \in \ell^{r}_{\xi}L^{p}_{\tau}$ we have
  \begin{equation}\label{eq:counter2}
    \sum_{\genfrac{}{}{0pt}{}{\xi,\xi_1,\xi_2 \in \Z}{\xi_1 \not= \xi,\xi_2\not=\xi}}
    \int
    \frac{\langle \xi \rangle^{\frac{1}{2}}f_0(\mu)f_1(\mu_1)f_2(\mu_2)
      |\xi_3|f_3(\mu_3)}
    {\langle \sigma_0\rangle^{-b}
      \langle \xi_1 \rangle^{\frac{1}{2}}\langle \sigma_1\rangle^{\frac{1}{2}}
      \langle \xi_2 \rangle^{\frac{1}{2}}\langle \sigma_2\rangle^{\frac{1}{2}}
      \langle \xi_3 \rangle^{\frac{1}{2}}\langle \sigma_3\rangle^{\frac{1}{2}}}
    d\tau d\tau_1 d\tau_2< \infty
  \end{equation}
  We choose
  \begin{align*}
    f_0(0,\tau)&=\chi(\tau) \text{ and } f_0(\xi,\tau)=0 \text{ for }\xi\not=0,\\
    f_2(1,\tau_2)&=\chi(\tau_2) \text{ and }
    f_2(\xi_2,\tau_2)=0 \text{ for } \xi_2\not=1,\\
    f_1(0,\tau_1)&=0 \text{ and } f_1(\xi_1,\tau_1)=\frac{\chi(\tau_1
      + (\xi_1 +1)^2)} {\langle \xi_1
      \rangle^{\frac{1}{4}}\ln^{\frac{1}{3}}(\lb \xi_1 \rb)}
    \text{ for } \xi_1\not=0,\\
    f_3(0,\tau_3)&=0 \text{ and } f_3(\xi_3,\tau_3)=\frac{\chi(\tau_3
      - \xi_3^2)} {\langle \xi_3
      \rangle^{\frac{1}{4}}\ln^{\frac{1}{3}}(\lb \xi_3 \rb)} \text{
      for } \xi_3\not=0,
  \end{align*}
  where $\chi$ denotes the characteristic function of $[-1,1]$. Then
  $f_0 \in \ell^{r}_{\xi}L^{p}_{\tau}$ and $f_2 \in
  \ell^{r'}_{\xi}L^{q'}_{\tau}$ for all $1\leq r,p,q \leq \infty$ and
  $f_{1},f_3 \in \ell^{r'}_{\xi}L^{q'}_{\tau}$ for all $r'\geq 4$ and
  $1\leq p,q \leq \infty$.  Let $ I(\xi_1)$ be defined as
  \begin{align*}
    &\int \chi(\tau) \chi(\tau_1 + (\xi_1 + 1)^2)\chi(\tau_2)
    \chi(\tau-\tau_1-\tau_2 - (\xi_1 + 1)^2) d\tau d\tau_1 d\tau_2\\
    \gs &\int \chi(\tau)\chi(\tau_1 + (\xi_1 + 1)^2)
    \chi(\tau -\tau_1- (\xi_1 + 1)^2)d\tau_1 d\tau\\
    \gs &\int \chi(\tau_1 + (\xi_1 + 1)^2) d\tau_1=2
  \end{align*}
  Due to $\lb \sigma_1 \rb^{\frac{1}{2}} \ls \lb \xi_1
  \rb^{\frac{1}{2}}$ the left hand side of (\ref{eq:counter2}) becomes
  \begin{equation*}
    \sum_{|\xi_1|\geq 1}
    \frac{I(\xi_1)}{\langle \xi_1 \rangle\ln^{\frac{2}{3}}(\lb \xi_1 \rb)}
    \\
    \gs \sum_{|\xi_1|\geq 1}
    \frac{1}{\langle \xi_1 \rangle\ln^{\frac{2}{3}}(\lb \xi_1 \rb)}
    = \infty
  \end{equation*}
  which contradicts (\ref{eq:counter2}).
\end{proof}

\section{The proof of the quintilinear estimate}\label{sect:quinti}
\noindent
Before we start with the proof of Theorem \ref{thm:quinti} we show the
following trilinear refinement of the $L^6$ Strichartz type
estimate, see \cite{Bo93a}, Proposition 2.36. The major point is, that
for one of the factors the loss of $\eps$ derivatives can be avoided.
In fact, this refinement also follows by carefully using the
decomposition arguments and the Galilean transformation in
\cite{Bo93a}, Section 5. However, we decided to present a proof based
on the representation
$\|u\|^3_{L^6_{xt}}=\|u^2\overline{u}\|_{L^2_{xt}}$ which we learned
from \cite{MV}, in combination with the estimates from Section
\ref{sect:nodc}. Similar arguments were already used in \cite{DPST06},
cf.  Proposition 4.6 and its proof.
\begin{lemma}\label{lem:stref}
  For $\frac{1}{3}<b<\frac{1}{2}$ and $s>3(\frac{1}{2}-b)$ the
  estimate
  \begin{equation}\label{eq:q1}
    \|u_1u_2\overline{u}_3\|_{L^2_{xt}} \ls \|u_1\|_{X^{s,b}}\|u_2\|_{X^{s,b}}\|u_3\|_{X^{0,b}}
  \end{equation}
  holds true.
\end{lemma}
\begin{proof}
  We rewrite $ u_1u_2\overline{u}_3=C_1(u_1,u_2,u_3)+C_2(u_1,u_2,u_3)
  $ for
  \begin{align*}
    \widehat{C_1(u_1,u_2,u_3)}(\xi)&=(2\pi)^{-1}
    \sum_{\genfrac{}{}{0pt}{}{\xi=\xi_1+\xi_2+\xi_3}{\xi_1 \not= \xi,
        \xi_2\not=\xi}} \widehat{u}_1(\xi_1)\widehat{u}_2(\xi_2)
    \widehat{\overline{u}}_3(\xi_3)\\
    \widehat{C_2(u_1,u_2,u_3)}(\xi)&=(2\pi)^{-1}
    \sum_{\genfrac{}{}{0pt}{}{\xi=\xi_1+\xi_2+\xi_3}{\xi_1 = \xi
        \text{ or } \xi_2=\xi}}
    \widehat{u}_1(\xi_1)\widehat{u}_2(\xi_2)
    \widehat{\overline{u}}_3(\xi_3)
  \end{align*}
  where we suppressed the $t$ dependence.  By Plancherel's identity we
  observe that
  \begin{equation*}
    C_2(u_1,u_2,u_3)=u_1\dashint_0^{2\pi}u_2\overline{u}_3 dy
    +u_2\dashint_0^{2\pi}u_1\overline{u}_3 dy
    -u_1\ast u_2 \ast \overline{u}_3
  \end{equation*}
  where $\ast$ denotes convolution with respect to normalized Lebesgue
  measure on $[0,2\pi]$. Clearly,
  \begin{align*}
    \left\|C_2(u_1,u_2,u_3)\right\|_{L^2_t L^2_{x}}\ls \prod_{1 \leq k
      \leq 3}\|u_k\|_{L^6_tL^2_x}\ls \prod_{1 \leq k \leq 3}
    \|u_k\|_{X^{0,\frac{1}{3}}}
  \end{align*}
  by Sobolev estimates in the time variable. For the convolution term
  we also used Young's inequality.  So it remains to prove
  (\ref{eq:q1}) with $\|u_1u_2\overline{u}_3\|_{L^2_{xt}}$ on the left
  hand side replaced by $\|C_1(u_1,u_2,u_3)\|_{L^2_{xt}}$.  Now by
  Cauchy-Schwarz' inequality and Fubini's theorem (cp. the arguments
  in the previous section) matters reduce to show that
  \[\sup_{\xi,\tau}\Sigma(\xi,\tau)<\infty\]
  where $\Sigma(\xi,\tau)$ is defined as
  \begin{equation*}
    \sum_{\genfrac{}{}{0pt}{}{\xi=\xi_1+\xi_2+\xi_3}
      {\xi_1 \not= \xi, \xi_2\not=\xi}}\langle \xi_1 \rangle^{-2s}\langle \xi_2 \rangle^{-2s}
    \int  \langle \tau_1 + \xi_1^2 \rangle^{-2b}
    \langle \tau_2 + \xi_2^2 \rangle^{-2b}\langle \tau_3 - \xi_3^2 \rangle^{-2b}d\tau_1d\tau_2.
  \end{equation*}
  Using Lemma \ref{lem:fam2.12} twice, we see that
  \begin{align*}
    \Sigma(\xi,\tau) & \ls
    \sum_{\genfrac{}{}{0pt}{}{\xi=\xi_1+\xi_2+\xi_3} {\xi_1 \not= \xi,
        \xi_2\not=\xi}}\langle \xi_1 \rangle^{-2s}\langle \xi_2
    \rangle^{-2s}
    \langle \tau + \xi^2 -2(\xi - \xi_1)(\xi - \xi_2)\rangle^{2-6b}& \\
    & \ls \left(\sum_{\genfrac{}{}{0pt}{}{\xi=\xi_1+\xi_2+\xi_3}
        {\xi_1 \not= \xi, \xi_2\not=\xi}}\langle \xi_1
      \rangle^{0-}\langle \xi_2 \rangle^{0-} \langle \tau + \xi^2
      -2(\xi - \xi_1)(\xi - \xi_2) \rangle^{\frac{2-6b}{1-2s}-}
    \right)^{(1-2s)+}&
  \end{align*}
  by H\"older's inequality. Since by assumption $\frac{2-6b}{1-2s}<
  -1$, a final application of Corollary \ref{cor:sums}, Part 3,
  completes the proof.
\end{proof}
In the $L^2_{xt}$-norm on the left hand side of (\ref{eq:q1}) we may,
of course, replace any single factor by its complex conjugate.
Especially we have
\begin{equation}\label{eq:q2}
  \|u_1u_2\overline{u}_3\|_{L^2_{xt}} \ls \|u_1\|_{X^{0,b}}\|u_2\|_{X^{s,b}}\|u_3\|_{X^{s,b}}
\end{equation}
Fixing $u_2$ and $u_3$ and considering the linear operator
\begin{equation*}
  X^{0,b} \to L^2_{xt}: \quad u_1 \mapsto u_1u_2\overline{u}_3
\end{equation*}
we obtain by duality the estimate
\begin{equation}\label{eq:q3}
  \|v\overline{u}_2u_3\|_{X^{0,-b}}\ls \|v\|_{L^2_{xt}}\|u_2\|_{X^{s,b}}\|u_3\|_{X^{s,b}}
\end{equation}
Choosing $v=u_1\overline{u}_4u_5$ and applying (\ref{eq:q2}) (and
(\ref{eq:q1}), respectively) again, we have shown the following
quintilinear estimate:
\begin{corollary}\label{cor:qui}
  Set $i=1$ or $i=4$.  For $\frac{1}{3}<b_0<\frac{1}{2}$ and
  $s_0>3(\frac{1}{2}-b_0)$ the estimate
  \begin{equation}\label{eq:q4}
    \|u_1 \overline{u}_2u_3\overline{u}_4 u_5\|_{X^{0,-b_0}} \ls 
    \|u_i\|_{X^{0,b_0}} \prod_{\genfrac{}{}{0pt}{}{k=1}{k\not=i}}^5\|u_k\|_{X^{s_0,b_0}}
  \end{equation}
  is valid.
\end{corollary}

In order to prove Theorem \ref{thm:quinti} we shall rely on
the interpolation properties of our scale of spaces obtained by the
complex method.
\begin{lemma}\label{lem:interpolation}
  Let $s_i,b_i \in \R$, $1<r_i,p_i<\infty$ for $i=1,2$. Then,
  \begin{equation*}
    (X^{s_0,b_0}_{r_0,p_0},X^{s_1,b_1}_{r_1,p_1})_{[\theta]}=X^{s,b}_{r,p} \qquad (0<\theta<1) 
  \end{equation*}
  where
  \begin{align*}
    s=(1-\theta)s_0+\theta s_1 \quad ,& \quad  b=(1-\theta)b_0+\theta b_1 \,  ,\\
    \frac{1}{r} =\frac{1-\theta}{r_0}+\frac{\theta}{r_1}\quad ,& \quad
    \frac{1}{p} =\frac{1-\theta}{p_0}+\frac{\theta}{p_1}\,.
  \end{align*}
\end{lemma}
\begin{proof}
  The map
  \begin{equation*}
    \mathcal{F} \circ e^{-it\partial_x^2}:\; X^{s,b}_{r,p} \to \ell^{r'}_\xi(\lb \xi\rb^s;L^{p'}_\tau(\lb\tau\rb^b))
  \end{equation*}
  is an isometric isomorphism. Here, the image space is the space of
  sequences in $\ell^{r'}_\xi$ with weight $\lb \xi\rb^s$, taking
  values in $L^{p'}_\tau$ with weight $\lb\tau\rb^b$ (with the natural
  norm). Now, arguing as in \cite{BeLo76} Theorem 5.6.3 (replacing
  $2^k$ by $k$) and using \cite{BeLo76} Theorem 5.5.3 the claim
  follows.
\end{proof}
\begin{proof}[Proof of Theorem \ref{thm:quinti}]
  By Sobolev-type embeddings and Young's inequality, we see that for
  $r_1,q_1>1$, $s_1>\frac{1}{q_1}$, $b_1>\frac{1}{3}$ and an auxiliary
  exponent $p$ with $b_1+\frac{1}{2}>\frac{1}{p}>5(\frac{1}{2}-b_1)$
  \begin{align}\label{eq:q5}
    \|u_1 \overline{u}_2 u_3 \overline{u}_4 u_5\|_{X^{0,-b_1}_{r_1,2}}
    & \ls
    \|u_1 \overline{u}_2 u_3 \overline{u}_4 u_5\|_{X^{0,0}_{r_1,p}} & \nonumber\\
    & \leq \|u_1\|_{X^{0,0}_{r_1,5p}} \prod_{i=2}^5
    \|u_i\|_{X^{0,0}_{\infty,5p}}&\\
    & \ls
    \|u_1\|_{X^{0,b_1}_{r_1,2}}\prod_{i=2}^5\|u_i\|_{X^{s_1,b_1}_{q_1,2}}\nonumber
  \end{align}
  Now fix $\frac{4}{3}<q\le r\le2$ and $b>\frac{1}{6}+\frac{1}{3q}$.
  We will use complex multilinear interpolation (\cite{BeLo76},
  Theorem 4.4.1) between (\ref{eq:q4}) and (\ref{eq:q5}) with
  interpolation parameter $\theta = \frac{1}{2}$. To do so, we choose
  \begin{equation*}
    s_0=\frac{3}{2}-\frac{2}{q}-\eps,b_0=\frac{2}{3q}+\eps
  \end{equation*}
  in the endpoint \eqref{eq:q4} and
  \begin{equation*}
    \frac{1}{r_1}=\frac{2}{r}-\frac{1}{2}\; , \; \frac{1}{q_1}=\frac{2}{q}-\frac{1}{2} \; , \;
    s_1=\frac{2}{q}-\frac{1}{2}+\eps \; , \; b_1=\frac{1}{3}+\eps
  \end{equation*}
  in the endpoint \eqref{eq:q5}, where
  $\eps:=b-\frac{1}{6}-\frac{1}{3q}>0$ and use Lemma
  \ref{lem:interpolation}. As a result,
  \begin{equation*}
    \|u_1 \overline{u}_2u_3\overline{u}_4 u_5\|_{X^{0,-b}_{r,2}} \ls
    \|u_1\|_{X^{0,b}_{r,2}} \prod_{i=2}^5\|u_i\|_{X^{\frac{1}{2},b}_{q,2}}
  \end{equation*}
  where in the last line of (\ref{eq:q5}) as well as in the last
  expression we may exchange $u_1$ and $u_4$. Now our first claim
  \eqref{eq:quinti_gen} follows from $\langle \xi \rangle \le
  \sum_{1\le i \le 5}\langle \xi_i \rangle$. We use (\ref{eq:x_emb})
  of Lemma \ref{lem:cont_emb}, that is
  $X^{\frac{1}{2},-\frac{1}{2}+}_{r,2} \subset
  X^{\frac{1}{2},-1}_{r,\infty}$, and apply (\ref{eq:gain_T}) with
  $\nu = \frac{1}{3q'}-$ to all the six norms appearing, which gives a
  factor $T^{\frac{2}{q'}-}$.  Finally, \eqref{eq:quinti_a} follows
  from \eqref{eq:quinti} and the proof of Theorem \ref{thm:quinti} is
  complete.
\end{proof}

\section{The gauge transform}\label{sect:gauge}
\noindent
In this subsection we study a nonlinear transformation which turned
out to be a key ingredient to the well-posedness theory of the DNLS.
This type of transformation for the DNLS was already used by
Hayashi-Ozawa \cite{HO92,Hay93,HO94}, later by Takaoka \cite{Tak99},
in the nonperiodic case and then adapted to the periodic setting by
the second author in \cite{H06,H06diss}. Let us define for $u\in
C([-T,T],L^2(\T))$
\begin{equation}\label{eq:gauge0}
  \mathcal{G}_0u:=\exp\left(-i\mathcal{I}u\right)u
\end{equation}
where
\begin{equation}\label{eq:I}
  \mathcal{I}u(t,x):=\dashint_0^{2\pi} \int_\theta^x |u(t,y)|^2
  -\dashint_0^{2\pi}|u(t,z)|^2dz dy d\theta
\end{equation}
is the unique primitive of
\begin{equation*}
  x\mapsto |u(t,x)|^2
  -\dashint_0^{2\pi}|u(t,z)|^2 dz
\end{equation*}
with vanishing mean value.
Before we study the mapping properties of this transformation let us
recall the Sobolev multiplication law in our setting.
\begin{lemma}\label{lem:sob_mult}
  Let $1<r<\infty$, $0\leq s\leq s_1$ and $s_1>\frac{1}{r}$.  Then,
  \begin{equation}\label{eq:sob_mult1}
    \|u_1u_2\|_{\widehat{H}^{s}_r}\ls
    \|u_1\|_{\widehat{H}^{s_1}_{r}}\|u_2\|_{\widehat{H}^{s}_{r}}
  \end{equation}
  In particular, for $1<r<\infty$, $s>\frac{1}{r}$ we have
  \begin{equation}\label{eq:sob_mult2}
    \|u_1u_2\|_{\widehat{H}^{s}_{r}}\ls
    \|u_1\|_{\widehat{H}^{s}_{r}}\|u_2\|_{\widehat{H}^{s}_{r}}
  \end{equation}
\end{lemma}
\begin{proof}
  We have $ \lb \xi \rb^s \ls \lb\xi-\xi_1\rb^{s}
  +\lb\xi_1\rb^{s_1}\lb\xi-\xi_1\rb^{s-s_1} $ and by Young's
  inequality
  \begin{align*}
    \|u_1u_2\|_{\widehat{H}^{s}_r} \ls & \|\widehat{u}_1\|_{\ell^1}
    \|\lb\xi \rb^s \widehat{u}_2\|_{\ell^{r'}}+
    \|\lb\xi\rb^{s_1}\widehat{u}_1\|_{\ell^{r'}}
    \|\lb\xi \rb^{s-s_1} \widehat{u}_2\|_{\ell^{1}}\\
    \ls & \|u_1\|_{\widehat{H}^{s_1}_r} \|u_2\|_{\widehat{H}^{s}_r}
  \end{align*}
  because $\|\lb\xi\rb^{-s_1} \widehat{u}_i\|_{\ell^{1}}\ls
  \|\widehat{u}_i\|_{\ell^{r'}}$, $i=1,2$.
\end{proof}
\begin{lemma}\label{lem:lip}
  Let $1< r \leq 2$ and $s>\frac{1}{r}-\frac{1}{2}$ or $r=2$ and
  $s\geq 0$. Then,
$$
\mathcal{G}_0: C([-T,T],\widehat{H}^{s}_r(\T)) \to
C([-T,T],\widehat{H}^{s}_r(\T))
$$
is a locally bi-lipschitz homeomorphism with inverse $
\mathcal{G}^{-1}_0 u=\exp\left(i\mathcal{I}u\right)u $.
\end{lemma}
\begin{proof}
  We will transfer the ideas from the proof of \cite{H06}, Lemma 2.3
  (where the claim is shown in the $L^2$ setting) to the case $1< r <
  2$ and $s>\frac{1}{r}-\frac{1}{2}$.  Obviously, it suffices to prove
  \begin{equation}\label{eq:exp}
    \begin{split}
      &\left\| \left(\exp(\pm i \mathcal{I}f)-\exp(\pm i
          \mathcal{I}g)\right)h
      \right\|_{\widehat{H}^{s}_r}\\
      \ls &
      \exp(c\|f\|^2_{\widehat{H}^s_r}+c\|g\|^2_{\widehat{H}^s_r})
      \|f-g\|_{\widehat{H}^s_r}\|h\|_{\widehat{H}^s_r}
    \end{split}
  \end{equation}
  for smooth $f,g,h \in \widehat{H}^s_r$.  By the series expansion of
  the exponential and \eqref{eq:sob_mult1} we infer that the left hand
  side of \eqref{eq:exp} is bounded by
  \begin{equation*}
    \|h\|_{\widehat{H}^s_r}\|\mathcal{I}f-\mathcal{I}g\|_{\widehat{H}^{s_1}_r}
    \sum_{n=1}^\infty \frac{c^{n}}{n!} \sum_{k=0}^{n-1}
    \|\mathcal{I}f\|^{k}_{\widehat{H}^{s_1}_r}\|\mathcal{I}g\|^{n-1-k}_{\widehat{H}^{s_1}_r}
  \end{equation*}
  for $s_1=\max\{\frac{1}{r}+,s\}$. By the definition of $\mathcal{I}$
  it follows
  \begin{equation*}
    \|\mathcal{I}f-\mathcal{I}g\|_{\widehat{H}^{s_1}_r}\ls
    \||f|^2-|g|^2\|_{\widehat{H}^{s_1-1}_r}
  \end{equation*}
  In the case $s\leq \frac{1}{r}$ we have
  \begin{equation*}
    \||f|^2-|g|^2\|_{\widehat{H}^{s_1-1}_r}\ls 
    (\|\widehat{f}\|_{\ell^{2-}}+\|\widehat{g}\|_{\ell^{2-}})\|\widehat{f}-\widehat{g}\|_{\ell^{2-}}
  \end{equation*}
  Because $s>\frac{1}{r}-\frac{1}{2}$ we find
  \begin{equation*}
    \|\mathcal{I}f-\mathcal{I}g\|_{\widehat{H}^{s_1}_r}
    \ls \left(\|f\|_{\widehat{H}^s_r}+\|g\|_{\widehat{H}^s_r}\right)
    \|f-g\|_{\widehat{H}^s_r}
  \end{equation*}
  In the case $s>\frac{1}{r}$ we use \eqref{eq:sob_mult2} to deduce
  \begin{equation*}\||f|^2-|g|^2\|_{\widehat{H}^{s_1-1}_r}\ls
    \|f-g\|_{\widehat{H}^{s}_r}
    (\|f\|_{\widehat{H}^{s}_r}+\|g\|_{\widehat{H}^{s}_r})
  \end{equation*}
  and the estimate \eqref{eq:exp} follows.
\end{proof}
The following lemma is contained in \cite{H06diss} in the $r=2$ case.
\begin{lemma}\label{lemma:trans}
  Let $1< r \leq 2$ and $s>\frac{1}{r}-\frac{1}{2}$ or $r=2$ and
  $s\geq 0$.  The translation operators
  \begin{equation*}
    \begin{split}
      \tau^\mp&: C([-T,T],\widehat{H}^s_r(\T)) \to C([-T,T],\widehat{H}^s_r(\T))\\
      \tau^{\mp}u(t,x)&:=u(t,x\mp 2t\dashint_0^{2\pi}|u(t,y)|^2dy)
    \end{split}
  \end{equation*}
  are continuous. However, their restrictions to arbitrarily small
  balls are not uniformly continuous.
\end{lemma}
\begin{proof}
  We only sketch the main ideas and refer to \cite{H06diss},
  Propositions 3.2.1 and 3.2.2 for details in the $H^s$ case which
  easily carry over to the $\widehat{H}^{s}_r$ setting: Because the
  embedding $\widehat{H}^{s}_r\subset L^2$ is continuous, the
  continuity statement follows from the continuity of the map
  \begin{equation*}
    \begin{split}
      \tau: \R\times C([-T,T],\widehat{H}^s_r(\T)) &\to C([-T,T],\widehat{H}^s_r(\T))\\
      \tau(h,u)(t,x)&:=u(t,x+ht)
    \end{split}
  \end{equation*}
  If we fix the time variable, the map
  \begin{equation*}
    \begin{split}
      \R\times \widehat{H}^s_r(\T) \to \widehat{H}^s_r(\T) \, , \quad
      (h,f)\mapsto f(\cdot+h)
    \end{split}
  \end{equation*}
  is continuous. This follows from the fact that a translation by a
  fixed amount is an isometry in $\widehat{H}^s_r(\T)$ combined with
  $e^{i\xi h}\to e^{i\xi h_0}$ for $h\to h_0$ and the dominated
  convergence theorem.  Now, because $[-T,T]$ is compact, we may
  approximate $u \in C([-T,T],\widehat{H}^s_r(\T))$ uniformly by a
  piecewise constant (in time) function and apply the result for $t$
  fixed.
  
  For $r>0$, the sequences of functions
  \begin{equation*}
    u_{n,j}(t,x)=u_{n,j}(x)=rn^{-s}e^{inx}+c_{n,j}\; , \quad n \in \N, \, j=1,2
  \end{equation*}
  with $c_{n,1}=\frac{1}{\sqrt{n}}$ and $c_{n,2}=0$ provide a
  counterexample to the uniform continuity on balls.
\end{proof}
As in \cite{H06,H06diss} we define $\mathcal{G}=\mathcal{G}_0\circ
\tau^{-}$, i.e.
\begin{equation}\label{eq:gauge}
  \mathcal{G}u(t,x)=(\mathcal{G}_0u)\left(t,x-2t\dashint_0^{2\pi}|u(t,y)|^2dy \right)
\end{equation}
\begin{lemma}\label{lemma:gauge_prop}
  Let $u,v\in C([-T,T],H^2)\cap C^1((-T,T),L^2)$ such that
  $v=\mathcal{G}u$. Then, $u$ solves \eqref{eq:dnls} if and only if
  $v$ solves
  \begin{equation}\label{eq:gauge_dnls}
    \begin{split}
      i \partial_t v(t) + \partial_x^2
      v(t)&=-i\mathcal{T}(v)(t)-\tfrac{1}{2} \mathcal{Q}(v)(t)\; ,
      \quad  t \in (-T,T)\\
      v(0)&=\mathcal{G}u(0)
    \end{split}
  \end{equation}
  where
  \begin{equation}\label{eq:mod_tri1}
    \mathcal{T}(v)=
    v^2\partial_x \bar{v}-2i\dashint_0^{2\pi} \Imag v\partial_x  \bar{v} \, dx \,v
  \end{equation}
  and
  \begin{equation}\label{eq:mod_quinti1}
    \mathcal{Q}(v)=\left(|v|^4-\dashint_0^{2\pi}|v|^4 dx \right)v
    -2\dashint_0^{2\pi}|v|^2dx
    \left(|v|^2-\dashint_0^{2\pi}|v|^2dx\right)v
  \end{equation}
  i.e. $\mathcal{T}(v)=T(v,v,v)$ and $\mathcal{Q}(v)=Q(v,v,v,v,v)$ for
  $T$ and $Q$ defined in \eqref{eq:mod_tri} and \eqref{eq:mod_quinti},
  respectively.  Moreover, the map
  \begin{equation*}
    \mathcal{G}: C([-T,T],\widehat{H}^s_r(\T)) \to
    C([-T,T],\widehat{H}^s_r(\T))
  \end{equation*}
  is a homeomorphism with inverse $ \mathcal{G}^{-1}=\tau^+\circ
  \mathcal{G}_0^{-1} $. The restrictions of $\mathcal{G}$ and
  $\mathcal{G}^{-1}$ to arbitrarily small balls fail to be uniformly
  continuous. However, $\mathcal{G}$ is locally bi-lipschitz on
  subsets of functions with prescribed $L^2$ norm.
\end{lemma}
\begin{proof}
  To see the equivalence of \eqref{eq:dnls} and \eqref{eq:gauge_dnls}
  the calculations for the periodic case may be found in \cite{H06},
  Section 2.  The fact that we may represent $\mathcal{T}$ and
  $\mathcal{Q}$ via convolution operators on the Fourier side where
  certain frequency interactions are cancelled out was remarked in
  \cite{H06diss}, Remark 3.2.7. The verification of the precise
  formulas are straightforward, using (suppressing the $t$ dependence)
  \begin{align*}
    (2\pi)^{-1}\widehat{v}\ast\widehat{\partial_x \overline{v}}(0)&=i
    \dashint_0^{2\pi}
    \Imag v \partial_x \overline{v} \, dx\\
    (2\pi)^{-1} \widehat{v}\ast\widehat{\overline{v}}(0)&=\dashint_0^{2\pi}|v|^2 \, dx\\
    (2\pi)^{-2}\widehat{v}\ast\widehat{\overline{v}}\ast\widehat{v}\ast\widehat{\overline{v}}(0)
    &=\dashint_0^{2\pi}|v|^4 \, dx
  \end{align*}
  The mapping properties
  follow from Lemma \ref{lem:lip}.
\end{proof}
\begin{remark}\label{rem:cancellation}
  The cancellation of certain frequency interactions due to the term
  $$
  2\dashint_0^{2\pi} \Imag v \partial_x \overline{v} \, dx v \, ,
  $$
  which is crucial for our arguments, cp. \eqref{eq:mod_tri}, is an important feature of the gauge
  transformation. We observe as well that this expression itself is not well-defined in
  $\widehat{H}^{\frac{1}{2}}_r(\T)$ for $1<r<2$.
\end{remark}
\section{Proof of well-posedness}\label{sect:proof_wp}
\noindent
Now, we show that the Cauchy problem \eqref{eq:gauge_dnls} is
well-posed.  Let $\chi \in C_0^\infty(\R)$ be nonnegative and
symmetric such that $\chi(t)=0$ for $|t|\geq 2$ and $\chi(t)=1$ for
$|t|\leq 1$.  Recall that $Z^s_r:=X^{s,\frac{1}{2}}_{r,2} \cap
X^{s,0}_{r,\infty}$.  We have, similar to the $L^2$ case, the
following linear estimates:
\begin{lemma}\label{lemma:lin_est}
  Let $s\in \R$, $1<r<\infty$.
  \begin{align}
    \left\|\chi Su_0\right\|_{Z^s_r}&\ls \|u_0\|_{\widehat{H}^{s}_r} \label{eq:lin_hom}\\
    \left\|\chi \int_0^t S(t-t')u(t') dt'\right\|_{Z^s_r}&\ls
    \|u\|_{X^{s,-\frac{1}{2}}_{r,2}}
    +\|u\|_{X^{s,-1}_{r,\infty}}\label{eq:lin_inhom}
  \end{align}
\end{lemma}
\begin{proof}
  We use the approach from \cite{CKSTT04}, Lemma 3.1.  Let $u_0 \in
  C^\infty(\R)$ be periodic. We calculate $\mathcal{F}(\chi S
  u_0)(\tau,\xi)=\mathcal{F}_t\chi(\tau+\xi^2)\mathcal{F}_x u_0(\xi)$
  and \eqref{eq:lin_hom} follows because $\mathcal{F}_t\chi$ is
  rapidly decreasing.  It suffices to consider smooth $u$ with
  $\supp(u)\subset \{(t,x)\mid |t|\leq 2\}$.  We rewrite
  \begin{equation*}
    \chi(t) \int_0^t S(t-t') f(t')\, dt'=F_1(t)+F_2(t)
  \end{equation*}
  where
  \begin{align*}
    F_1(t)
    =&\frac{1}{2}\chi(t)S(t)\int_\R \varphi(t') S(-t') u(t')\, dt'\\
    F_2(t)=&\frac{1}{2}\chi(t)\int_\R \varphi(t-t')S(t-t') u(t')\, dt'
  \end{align*}
  and $\varphi(t')=\chi(t'/10)\sign(t')$. Concerning $\varphi$ we have
  \begin{equation}
    \label{eq:varphi}
    |\mathcal{F}_t\varphi(\tau)| \ls \lb\tau\rb^{-1}
  \end{equation}
  Estimate \eqref{eq:lin_hom} yields
  \begin{equation*}
    \|F_1\|_{Z^s_r} \ls
    \left\|\int_\R \varphi(t') S(-t') u(t')\, dt'\right\|_{\widehat{H}^{s}_r}
  \end{equation*}
  Parseval's equality implies
  \begin{equation*}
    \mathcal{F}_x \left(\int_\R \varphi(t') S(-t') u(t')\, dt'\right)(\xi)
    =\int_\R \overline{\mathcal{F}_t \varphi}(\tau+\xi^2) \mathcal{F}u(\tau,\xi) \,
    d\tau
  \end{equation*}
  which gives
  \begin{equation*}
    \left\|\int_\R \varphi(t') S(-t') u(t')\, dt'\right\|_{\widehat{H}^{s}_r} \ls
    \|u\|_{X^{s,-1}_{r,\infty}}
  \end{equation*}
  by \eqref{eq:varphi}. Now, let us consider $F_2$.  Due to Young's
  inequality, we may remove the cutoff function $\chi$ in front of the
  integral. The Fourier transform of the remainder is given by
  \begin{equation*}
    \mathcal{F}\left(\int_\R \varphi(t-t')S(t-t') u(t')\, dt'\right) (\tau,\xi)
    = \mathcal{F}_t \varphi(\tau+\xi^2) \mathcal{F}{u}(\tau,\xi)
  \end{equation*}
  Estimate \eqref{eq:varphi} proves \eqref{eq:lin_inhom}.
\end{proof}
A standard application of the fixed point argument gives
\begin{theorem}\label{thm:wp_gauge_dnls}
  Let $\frac{4}{3}<q \leq r\leq 2$. Then, for every
  \begin{equation*}
    v_0\in B_R:=\{ v_0 \in \widehat{H}^{\frac{1}{2},r}(\T)\mid
    \|v_0\|_{\widehat{H}^{\frac{1}{2},q}}< R\}
  \end{equation*}
  and $T \ls R^{-2q'-}$ there exists a solution
  \begin{equation*}
    v \in Z^{\frac{1}{2}}_{r}(T) \subset C([-T,T],\widehat{H}^{\frac{1}{2}}_r(\T))
  \end{equation*}
  of the Cauchy problem \eqref{eq:gauge_dnls}. This solution is unique
  in the space $Z^{\frac{1}{2}}_{q}(T)$ and the map
  \begin{equation*}
    \left(B_R, \|\cdot\|_{\widehat{H}^{\frac{1}{2},r}}\right)
    \longrightarrow C([-T,T],\widehat{H}^{\frac{1}{2},r}(\T)):
    \quad v_0 \mapsto v
  \end{equation*}
  is locally Lipschitz continuous. Moreover, it is real analytic.
\end{theorem}
\begin{proof}[Sketch of proof]
  As a consequence of the estimates \eqref{eq:lin_inhom},
  \eqref{eq:tri_x1}, \eqref{eq:tri_x2} and \eqref{eq:quinti_a}
  \begin{equation*}
    \Phi(v)(t):=\int_0^t e^{i(t-t')\partial_x^2}
    \left(-\tfrac{1}{2}\mathcal{Q}-i\mathcal{T}\right)v(t') \, dt'
  \end{equation*}
  extends to a continuous map $\Phi:Z^{\frac{1}{2}}_{r}(T)\to
  Z^{\frac{1}{2}}_{r}(T)$ for $\frac{4}{3} < r \leq 2$ along with the
  estimate
  \begin{equation}\label{eq:diff_est}
    \begin{split}
      \|\Phi(v_1)-\Phi(v_2)\|_{Z^{\frac{1}{2}}_{r}(T)} \ls &
      T^{\frac{1}{r'}-}
      \left(\|v_1\|_{Z^{\frac{1}{2}}_{r}(T)}+\|v_2\|_{Z^{\frac{1}{2}}_{r}(T)}\right)^2
      \|v_1-v_2\|_{Z^{\frac{1}{2}}_{r}(T)}\\
      +&T^{\frac{2}{r'}-}\left(\|v_1\|_{Z^{\frac{1}{2}}_{r}(T)}+\|v_2\|_{Z^{\frac{1}{2}}_{r}(T)}
      \right)^4\|v_1-v_2\|_{Z^{\frac{1}{2}}_{r}(T)}
    \end{split}
  \end{equation}
  and with \eqref{eq:lin_hom} we also have
  \begin{equation*}
    \|e^{it\partial_x^2} v_0+\Phi(v)\|_{Z^{\frac{1}{2}}_{r}(T)}
    \ls
    \|v_0\|_{\widehat{H}^{\frac{1}{2},r_1}}+T^{\frac{1}{r'}-}\|v\|^3_{Z^{\frac{1}{2}}_{r}(T)}
    +T^{\frac{2}{r'}-}\|v\|^5_{Z^{\frac{1}{2}}_{r}(T)}
  \end{equation*}
  Hence, for fixed $v_0$ the operator $ e^{it\partial_x^2} v_0+\Phi: D
  \to D $ is a strict contraction in some closed ball $D\subset
  Z^{\frac{1}{2}}_{r}(T)$ for small enough $T$. By the contraction
  mapping principle we find a fixed point $v \in
  Z^{\frac{1}{2}}_{r}(T)$ which is a solution of \eqref{eq:gauge_dnls}
  for small times. Similarly, the implicit function theorem shows that
  the map $v_0 \mapsto v$ is real analytic, hence locally Lipschitz.
  Uniqueness in $Z^{\frac{1}{2}}_{q}(T)$ follows by contradiction: A
  translation in time reduces matters to uniqueness for an arbitrarily
  short time interval which follows from the estimate
  \eqref{eq:diff_est} with $r=q$.  The lower bound on the maximal time
  of existence is a consequence of the mixed estimate
  \begin{equation*}
    \|v\|_{Z^{\frac{1}{2}}_{r}(T)}
    \ls \|v(0)\|_{\widehat{H}^{\frac{1}{2}}_{r}}
    +T^{\frac{1}{q'}-}\|v\|^2_{Z^{\frac{1}{2}}_{q}(T)}\|v\|_{Z^{\frac{1}{2}}_{r}(T)}
    +T^{\frac{2}{q'}-}\|v\|^4_{Z^{\frac{1}{2}}_{q}(T)}\|v\|_{Z^{\frac{1}{2}}_{r}(T)}
  \end{equation*}
  for solutions $v$ and an iteration argument.
\end{proof}
By combining Theorem \ref{thm:wp_gauge_dnls} with Lemma
\ref{lemma:gauge_prop} and some approximation arguments Theorem
\ref{thm:main} follows (this is carried out in detail in \cite{G05}
for the nonperiodic case).
\begin{remark}\label{rem:sharp_uni}
  In fact, our estimates also imply uniqueness of solutions of
  \eqref{eq:gauge_dnls} in a restriction space based on the
  $X^{\frac{1}{2},\frac{1}{2}}_{q,2}$ component only.  Hence, the
  optimal uniqueness statement concerning solutions of \eqref{eq:dnls}
  provided by our methods is the following: Let $4/3<q\leq 2$ and
  $u_1,u_2 \in C([-T,T],\widehat{H}^{\frac{1}{2}}_q(\T))$ be solutions
  of \eqref{eq:dnls} with $u_1(0)=u_2(0)$. If additionally
  $\mathcal{G}u_1,\mathcal{G}u_2 \in
  X^{\frac{1}{2},\frac{1}{2}}_{q,2}$ which also satisfy the equation
  \eqref{eq:gauge_dnls}, then $u_1=u_2$.
\end{remark}
\bibliographystyle{hplain}
\bibliography{literatur_dnls}\label{sect:refs}

\end{document}